\documentclass[a4paper,12pt]{article}
\usepackage[utf8]{inputenc}

\usepackage{mathrsfs}
\usepackage{graphicx}
\usepackage{enumerate}
\usepackage{multicol}
\usepackage{color}
\usepackage{amsmath,amssymb,amscd}
\usepackage{amsthm}
\usepackage{comment}
\usepackage{breakcites}

\usepackage{pdfpages}
\usepackage{hyperref}

\usepackage{times}
\usepackage[varg]{txfonts}

\usepackage[top=1in, bottom=1in, left=0.75in, right=0.75in]{geometry}

\theoremstyle{plain}
\newtheorem{thm}{Theorem}
\newtheorem{lem}{Lemma}
\newtheorem{cor}{Corollary}
\newtheorem{prop}{Proposition}

\theoremstyle{definition}
\newtheorem{dfn}{Definition}

\theoremstyle{remark}
\newtheorem{rem}{Remark}

\title{On Iverson's law of similarity}
\author{Eszter Gselmann, Christopher W.~Doble, and Yung-Fong Hsu}

\begin{document}

\maketitle

\begin{abstract}
\cite{Iverson2006b} proposed the law of similarity
\[
\xi_{s}(\lambda x)= \gamma(\lambda, s)\xi_{\eta(\lambda, s)}(x)
\]
for the sensitivity functions $\xi_{s}\, (s\in S)$. Compared to the former models, the generality of this one lies in that here $\gamma$ and $\eta$ can also depend on the variables $\lambda$ and $s$. In the literature, this model (or its special cases) is usually considered together with a given psychophysical representation (e.g. Fechnerian, subtractive, or affine). Our goal, however, is to study at first Iverson's law of similarity on its own. We show that if certain mild assumptions are fulfilled, then $\xi$ can be written in a rather simple form containing only one-variable functions.
The obtained form proves to be very useful when we assume some kind of representation. 

Motivated by \cite{HsuIverson2016}, we then study the above model assuming that the mapping $\eta$ is multiplicatively translational. 
First, we show how these mappings can be characterized. 
Later  we turn to the examination of Falmagne's power law. 
According to our results, the corresponding function $\xi$ can have a Fechnerian representation, and also it can have a subtractive representation. We close the paper with the study of the shift invariance property. 
\end{abstract}


\section{Introduction}

\cite{Iverson2006b} proposed the similarity model
\begin{equation}\label{iver_law}
 \xi_{s}(\lambda x)= \gamma(\lambda, s)\xi_{\eta(\lambda, s)}(x)\,,
\end{equation}
which we call \emph{Iverson's law of similarity}.  As described shortly, this model may be applied to a number of phenomena in psychophysics and experimental psychology.
The meaning of the notation in the model depends on the particular experimental context, examples of which follow.
The purpose of the present work is to further the study of Iverson's law of similarity on theoretical grounds. 
Models in psychology and other empirical sciences are often formalized by equations containing unknown functions, so-called functional equations. Due to this, many results in psychophysical theory can be
obtained through applications of functional equations, see \cite{Aczel1966},  \cite{Aczel1987}, \cite{AczelFalmganeLuce2000}, \cite{Falmagne1985}. 
A seminal example of this is the work in \cite{LuceEdwards1958} in uncovering the possible 'sensory scales' relating physical stimuli with the sensation strengths they elicit (see also  \cite{Iverson2006a}). Another illustrative example is the so-called Plateau's midgray experiment (briefly described below) \cite{Plateau1872}, \cite{Heller2006}.  
Regarding the theory of functional equations, we will primarily rely on J.~Aczél's monographs \cite{Aczel1966} and \cite{Aczel1987} and J.-C.~Falmagne's monograph \cite{Falmagne1985}. In addition, we will use results from the articles \cite{Aczel2005} and \cite{Lundberg1977}. These statements (see Lemmas \ref{lemma_Aczel} and \ref{lem_Lundberg}) are presented at the end of the second section.

The reader may wish to keep in mind the following experimental contexts for Iverson's law of similarity \eqref{iver_law} when considering the functional equation results in this paper.  
In one experimental context,  the participant is asked to judge  which of two stimuli  has the greater sensory impact (i.e., is louder, is heavier, is brighter, etc.).  The value $x$ is a non-negative number representing an intensity, and the value $\xi_{s}(x)$ indicates the intensity that is judged to be greater than intensity $x$ according to the measure of discriminability $s$. Depending on the experimental paradigm used, $s$ may be, for example, a probability, or a value proportional to $d^\prime$ in signal detection theory  \cite{GreenSwets1966}, or some other measure of   likelihood.\footnote{\label{foot_gen} In, for example,  \cite{Iverson2006b}, $s$ has the following meaning. Let  $P(x,y)$ be the probability that intensity $y$ is judged to be greater than intensity $x$.  If such probabilities are assumed to follow a model such as $P(x,y)=F[u(y)-u(x)]$, where $F$ is strictly monotonic, then fixing $P(x,y)=\pi$, we can write that model as $F^{-1}(\pi)= u(y)-u(x)$. But note that with $\pi$ fixed, we have that $y$ in this equation depends on both $x$ and $\pi$, so we write $y=\tilde{\xi_\pi}(x)$, that is, $\tilde{\xi_\pi}(x)$ is the intensity judged greater than $x$ with probability $\pi$. It is convenient to define $F^{-1}(\pi)=s$ and also $\tilde{\xi_\pi}(x)= \xi_s (x)$.}   The value $\lambda$ is a non-negative real number, and the subscript $\eta(\lambda, s)$, like $s$, measures the discriminability between stimuli; the function $\eta$ may depend on $s$ or $\lambda$ (or both).   The multiplier $\gamma(\lambda, s)$ similarly may depend on one or both of $s$ and $\lambda$.  In this experimental context, Iverson's law of similarity generalizes `Weber's law', for which the equation
\[\xi_{s}(\lambda x)=  \lambda \,\xi_s(x)
\]
 holds (\cite{LuceGalanter1963}), and it also generalizes  the `Falmagne's power law', for which 
 \[
 \xi_{s}(\lambda x)=  \lambda^{\phi(s)}\xi_s (x)
 \]
holds. (We use the label `Falmagne's power law' following \cite{Iverson2006b}.)    Falmagne's power law gives the so-called `near-miss to Weber's law' (\cite{McGillGoldberg1968}) when the exponent $\phi(s)$ is close to but different from $1$ for some or all values of $s$  (see also \cite{Doble_et_al2006,Augustin2008}).  
The near-miss to Weber's law model has been applied to data from a number of experimental situations, including pure-tone intensity discrimination \cite{JesteadtWierGreen1977}, 
 \cite{OsmanTzuoTzuo1980},  \cite{Florentine1986}, \cite{Florentineetal1987}, \cite{Doble_et_al2006}, and many others), visual area discrimination  (\cite{AugustinRoscher2008}), and even sugar concentration discrimination in nectar-feeding animals (\cite{Nachev_et_al2013}).  

In another experimental context, an auditory tone is embedded in a broadband noise background.  The participant's task is to match the perceived loudness of this tone/noise pairing with the loudness of an unmasked tone (a tone presented in quiet). This is the partial masking experimental context of \cite{PavelIverson1981}. If we write $\xi_s (x)$  to represent the intensity of the unmasked tone that matches the loudness of a tone of intensity $x$ embedded in a background of intensity $s$, \cite{PavelIverson1981} found that the data suggest the `shift invariance'  relationship 
\[
\xi_{\lambda^\theta s}(\lambda x) = \lambda\, \xi_s (x),
\]
where $x$ and $\lambda$ vary in respective non-negative intervals, and  $\theta$ takes values in the interval $]0, 1[$.  This equation is another specialization of Iverson's law of similarity.  

There are still other experimental contexts for which specializations of Iverson's law of similarity may be applied.  For example, in the midgray experiment by \cite{Plateau1872} alluded to above, participants were asked to paint a gray disk midway between a given white disk and a given black disk.  As described in \cite{Falmagne1985} and \cite{Heller2006}, the data suggest the relationship $\xi_{\lambda s} (\lambda x) = \lambda\, \xi_s (x)$, where $\xi_s (x)$ is the luminance of a disk judged to be midway between a disk of luminance $x$ and another of luminance $s$, and $\lambda$ is a positive number.  Further examples of the application of specializations of  Iverson's law of similarity are described in \cite{HsuIversonDoble2010} and \cite{HsuIverson2016}.
 
Given the above phenomena and their associated specializations of Iverson's law of similarity, there is the natural question of the usefulness of generalizing to equation \eqref{iver_law}; after all, if the specializations of equation \eqref{iver_law} are appropriate models for the phenomena, why would generalizations be of interest?  For the case of the near-miss to Weber's law for pure-tone intensity discrimination, generalizing to equation \eqref{iver_law} is of practical use because models such as 
\[
\xi_{s}(\lambda x)=  \lambda^{\phi(s)}\xi_s (x)
\]
are known to fail at low intensities (\cite{ViemeisterBacon1988}, \cite{WakefieldViemeister1990}, \cite{SchroderViemeisterNelson1994}). More
general candidate models are needed, and the study of equation \eqref{iver_law} can guide experimental work in distinguishing competing models; see  \cite[Page 289]{Iverson2006b}  in this regard.   In fact, deviations from Weber's law are so common across sensory modalities (see, for example, \cite{CarriotCullenChacron2021}) that similar comments may be made for a number of other intensity discrimination contexts beside those involving auditory tones.  Studying generalizations of models that have been successfully applied is a theme of the present work; for instance, in Section 3, we study a generalization of the transformation $\eta (\lambda, s)=\lambda^{-\theta} s$ appearing in the `shift invariance' relationship above, namely, we study the generalization of a `multiplicatively translational' $\eta$, 
\[
\eta(\lambda\tilde{\lambda}, s)= \eta(\tilde{\lambda}, \eta(\lambda, s))\,.
\]

Previously in the study of the possible functional forms in equation \eqref{iver_law} and its specializations, researchers assumed additional representations for the function $\xi_s$.  This approach stems historically from the original work by \cite{Fechner}, but it also has the purpose of helping to narrow down the possible solutions to the functional equations; without further assumptions about the functions $\xi_s$, $\eta$, and $\gamma$, equation \eqref{iver_law} allows for too many solutions for these functions. As examples of studies of (specializations of) equation \eqref{iver_law} in conjunction with additional representations for $\xi_s$, we find \cite{Falmagne1994}, who studied the near-miss model $\xi_{s}(\lambda x)=  \lambda^{\phi(s)}\xi_s (x)$ while also assuming a `subtractive representation' \[s=u(\xi_s(x))-w(x),\] \cite{Iverson2006b}, who studied equation \eqref{iver_law} in the context of a `Fechnerian representation'  
\[
s=u(\xi_s(x))-u(x), 
\]
\cite{FalmagneLundberg2000}, who studied the near-miss model assuming a `gain-control representation' 
\[
s=\frac{u(\xi_s (x))-w(x)}{\sigma(x)}
\]
\cite{IversonPavel1981}, who studied the shift invariance relationship 
\[
\xi_{\lambda^\theta s}(\lambda x) = \lambda\, \xi_s (x)
\]
under a similar gain-control representation; \cite{HsuIverson2016}, who studied the model 
\[
\xi_{s}(\lambda x)=  \lambda^{\phi(s)}\xi_{\eta(\lambda, s)} (x)
\]
under a `a gain-control representation' of the form  
\[s=\frac{u(\xi_s (x))-u(x)}{\sigma(x)}, \] 
and \cite{DobleHsu2020}, who studied equation \eqref{iver_law} assuming a subtractive representation.  As will be seen, in the present paper we examine representations such as these (Fechnerian, subtractive, and gain-control), but we also consider equation \eqref{iver_law} on its own, especially with assumptions about the function $\eta$.

The present work is arranged as follows.  Following \cite{Falmagne1985}, Section 2 contains the most basic concepts such as `psychometric families', `sensitivity functions' and some of their properties (e.g. `anchored', `balanced' and `parallel families'). Here we also present some representations known in psychophysics and their interpretations (e.g. Fechnerian, subtractive, affine, gain-control). Our results can be found in Section 3, which we divide into two subsections. 
First, we study equation \eqref{iver_law}, purely from the point of view of the theory of functional equations. In Remark \ref{form_xi}, we show that if there is at least one $s$ for which the function $\lambda\longmapsto \eta(\lambda, s)$ is invertible, then $\xi$ can written with the aid of one-variable functions in the form $\xi_{s}(x)= \dfrac{\Phi(f(s)x)}{g(s)}$. As a supplement to this, in Theorem \ref{thm_constant} we prove that if there exists a subset $\tilde{S}\subset S$ such that for all $s\in \tilde{S}$, the mapping 
$
 J\ni \lambda \longmapsto \eta(\lambda, s)
$
is constant, then $\xi$ also has a special form, namely 
$
 \xi_{s}(x)= \kappa(s)x^{\rho(s)}. 
$
The form obtained for the function $\xi$ in Remark \ref{form_xi} proves to be very useful when we assume some kind of representation. In line with this, we first describe those functions $\xi$ that have a certain gain-control representation (see Proposition \ref{prop1}), and then we examine equation \eqref{iver_law} again in the case where $\xi$ admits a representation 
\[
 \xi_{s}(x)= u(x)+v(s)
\]
(see Proposition \ref{prop3}). This latter representation is closely related to the so-called   parallel families. In Proposition \ref{prop4}, we provide a complete description of  anchored, balanced and parallel families. {These results are summarized in Part A of the graphical abstract.}

Motivated by \cite{HsuIverson2016}, in the second part of the third section we study \eqref{iver_law} assuming that the mapping $\eta$ is multiplicatively translational. 
First, we show that multiplicatively translational mappings can be written in the form
\[
 \eta(\lambda, s)= H(\lambda \cdot H^{-1}(s))
\]
if certain mild conditions are met, {see Part B of the graphical abstract.} Using this form, we examine what special form the function  $\gamma$ in \eqref{iver_law} can then take. Later we turn to the examination of Falmagne's power law, i.e. 
\[
 \xi_{s}(\lambda x)= \lambda^{\phi(s)}\xi_{\eta(\lambda, s)}(x)
\]
which is an important special case of \eqref{iver_law}. According to Remark \ref{cor1}, 
\[
 \xi_{s}(x)= x^{\phi(s)}F(x\cdot H^{-1}(s))
\]
is satisfied. After that, we show that if the function $\phi$ in the exponent is monotonic, then $\phi$ is necessarily locally constant.
According to Proposition \ref{prop_subtractive}, the corresponding function $\xi$ then has a substractive representation. As an application of the results of the subsection, we close the paper with the study of the shift invariance 
\[
 \xi_{s}(\lambda x)= \gamma(\lambda, s)\xi_{\lambda^{\theta} s}(x). 
\]
{These latter results are summarized in Part C of the graphical abstract.}

\begin{center}
  \includegraphics[height=8cm]{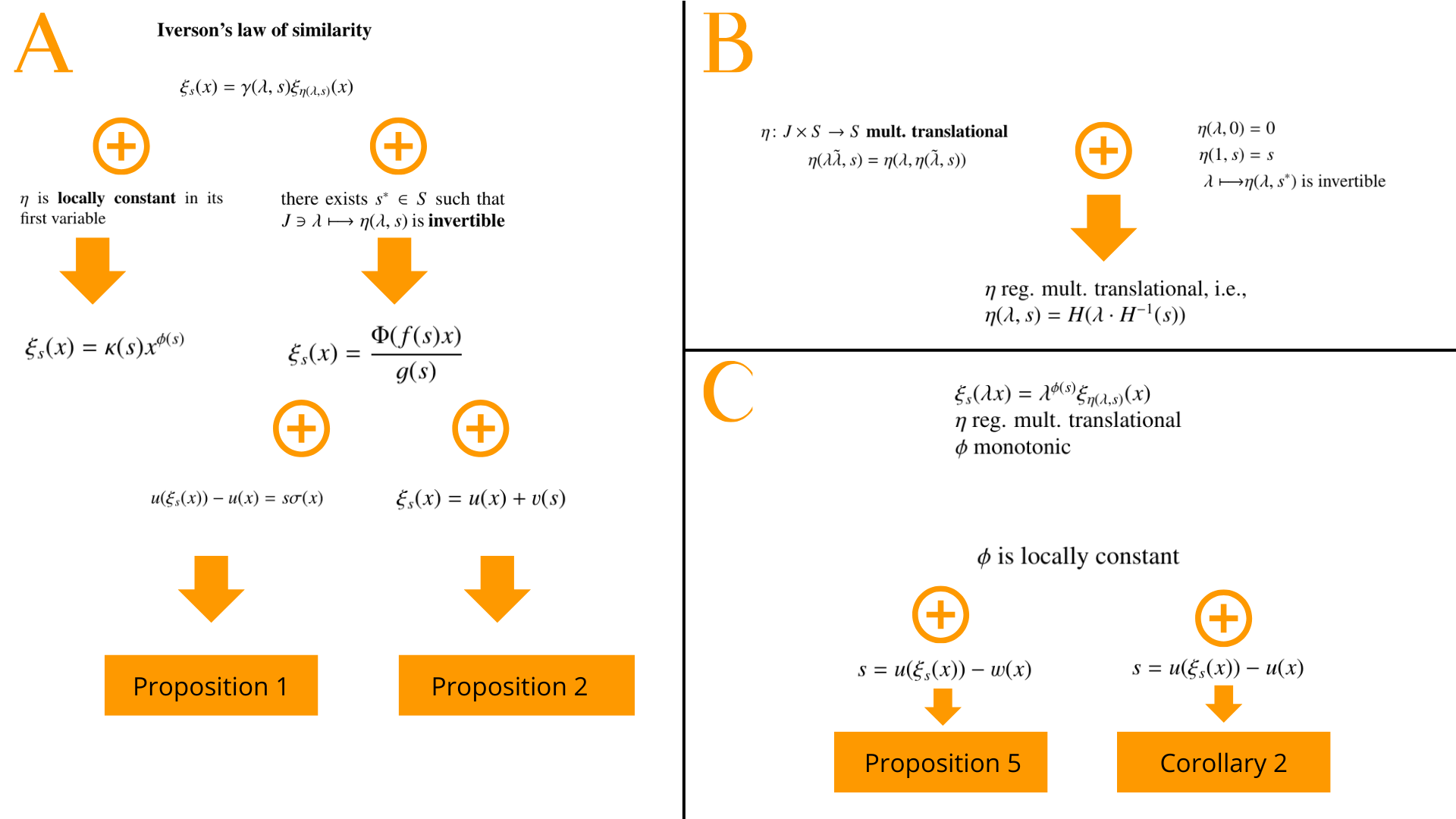}
 \end{center}

\section{Preliminaries}

\cite{Iverson2006b} proposed the similarity model
\begin{equation}\label{iverson_law}
 \xi_{s}(\lambda x)= \gamma(\lambda, s)\xi_{\eta(\lambda, s)}(x)\,,
\end{equation}
which we call \emph{Iverson's law of similarity}.  As described shortly, this model may be applied to a number of phenomena in psychophysics and experimental psychology.
The meaning of the notation in the model depends on the particular experimental context, examples of which follow.
The purpose of the present work is to further the study of Iverson's law of similarity on theoretical grounds. 
Models in psychology and other empirical sciences are often formalized by equations containing unknown functions, so-called functional equations. Due to this, many results in psychophysical theory can be
obtained through applications of functional equations, see \cite{Aczel1966},  \cite{Aczel1987}, \cite{AczelFalmganeLuce2000}, \cite{Falmagne1985}. 
A seminal example of this is the work in \cite{LuceEdwards1958} in uncovering the possible `sensory scales' relating physical stimuli with the sensation strengths they elicit \cite{Iverson2006a}. Another illustrative example is the so-called \cite{Plateau1872}'s midgray experiment (briefly described below) \cite{Heller2006}.  
Regarding the theory of functional equations, we will primarily rely on J.~Aczél's monographs \cite{Aczel1966,Aczel1987} and J.-C.~Falmagne's monograph \cite{Falmagne1985}. In addition, we will use results from the articles \cite{Aczel2005} and \cite{Lundberg1977}. These statements (see Lemmas \ref{lemma_Aczel} and \ref{lem_Lundberg}) are presented at the end of the second section.

The reader may wish to keep in mind the following experimental contexts for Iverson's law of similarity \eqref{iverson_law} when considering the functional equation results in this paper.  
In one experimental context,  the participant is asked to judge  which of two stimuli  has the greater sensory impact (i.e., is louder, is heavier, is brighter, etc.).  The value $x$ is a non-negative number representing an intensity, and the value $\xi_{s}(x)$ indicates the intensity that is judged to be greater than intensity $x$ according to the measure of discriminability $s$. Depending on the experimental paradigm used, $s$ may be, for example, a probability, or a value proportional to $d^\prime$ in signal detection theory  \cite{GreenSwets1966}, or some other measure of   likelihood.\footnote{\label{footnote} In, for example,  \cite{Iverson2006b}, $s$ has the following meaning. Let  $P(x,y)$ be the probability that intensity $y$ is judged to be greater than intensity $x$.  If such probabilities are assumed to follow a model such as $P(x,y)=F[u(y)-u(x)]$, where $F$ is strictly monotonic, then fixing $P(x,y)=\pi$, we can write that model as $F^{-1}(\pi)= u(y)-u(x)$. But note that with $\pi$ fixed, we have that $y$ in this equation depends on both $x$ and $\pi$, so we write $y=\tilde{\xi_\pi}(x)$, that is, $\tilde{\xi_\pi}(x)$ is the intensity judged greater than $x$ with probability $\pi$. It is convenient to define $F^{-1}(\pi)=s$ and also $\tilde{\xi_\pi}(x)= \xi_s (x)$.}   The value $\lambda$ is a non-negative real number, and the subscript $\eta(\lambda, s)$, like $s$, measures the discriminability between stimuli; the function $\eta$ may depend on $s$ or $\lambda$ (or both).   The multiplier $\gamma(\lambda, s)$ similarly may depend on one or both of $s$ and $\lambda$.  In this experimental context, Iverson's law of similarity generalizes `Weber's law', for which the equation
\[\xi_{s}(\lambda x)=  \lambda \,\xi_s(x)
\]
 holds \cite{LuceGalanter1963}, and it also generalizes `Falmagne's power law', for which 
 \[
 \xi_{s}(\lambda x)=  \lambda^{\phi(s)}\xi_s (x)
 \]
holds. (We use the label `Falmagne's power law' following \cite{Iverson2006b}.)    Falmagne's power law gives the so-called `near-miss to Weber's law' \cite{McGillGoldberg1968} when the exponent $\phi(s)$ is close to but different from $1$ for some or all values of $s$  \cite{Doble_et_al2006,Augustin2008}.  
The near-miss to Weber's law model has been applied to data from a number of experimental situations, including pure-tone intensity discrimination \cite{JesteadtWierGreen1977,OsmanTzuoTzuo1980,Florentine1986,Florentineetal1987,Doble_et_al2006}, visual area discrimination \cite{AugustinRoscher2008}, and even sugar concentration discrimination in nectar-feeding animals \cite{Nachev_et_al2013}.  

In another experimental context, an auditory tone is embedded in a broadband noise background.  The participant's task is to match the perceived loudness of this tone/noise pairing with the loudness of an unmasked tone (a tone presented in quiet). This is the partial masking experimental context of \cite{PavelIverson1981}. If we write $\xi_s (x)$  to represent the intensity of the unmasked tone that matches the loudness of a tone of intensity $x$ embedded in a background of intensity $s$, \cite{PavelIverson1981} found that the data suggest the `shift invariance'  relationship 
\[
\xi_{\lambda^\theta s}(\lambda x) = \lambda\, \xi_s (x),
\]
where $x$ and $\lambda$ vary in respective non-negative intervals, and  $\theta$ takes values in the interval $]0, 1[$.  This equation is another specialization of Iverson's law of similarity.  

There are still other experimental contexts for which specializations of Iverson's law of similarity may be applied.  For example, in the midgray experiment by \cite{Plateau1872} alluded to above, participants were asked to paint a gray disk midway between a given white disk and a given black disk.  As described in \cite{Falmagne1985} and \cite{Heller2006}, the data suggest the relationship $\xi_{\lambda s} (\lambda x) = \lambda\, \xi_s (x)$, where $\xi_s (x)$ is the luminance of a disk judged to be midway between a disk of luminance $x$ and another of luminance $s$, and $\lambda$ is a positive number.  Further examples of the application of specializations of  Iverson's law of similarity are described in \cite{HsuIversonDoble2010} and \cite{HsuIverson2016}.
 
Given the above phenomena and their associated specializations of Iverson's law of similarity, there is the natural question of the usefulness of generalizing to equation \eqref{iverson_law}; after all, if the specializations of equation \eqref{iverson_law} are appropriate models for the phenomena, why would generalizations be of interest?  For the case of the near-miss to Weber's law for pure-tone intensity discrimination, generalizing to Equation \eqref{iverson_law} is of practical use because models such as 
\[
\xi_{s}(\lambda x)=  \lambda^{\phi(s)}\xi_s (x)
\]
are known to fail at low intensities \cite{ViemeisterBacon1988,WakefieldViemeister1990,SchroderViemeisterNelson1994}. More
general candidate models are needed, and the study of Equation \eqref{iverson_law} can guide experimental work in distinguishing competing models; see page 289 of \cite{Iverson2006b} in this regard.   In fact, deviations from Weber's law are so common across sensory modalities \cite{CarriotCullenChacron2021} that similar comments may be made for a number of other intensity discrimination contexts beside those involving auditory tones.  Studying generalizations of models that have been successfully applied is a theme of the present work; for instance, in Section 3, we study a generalization of the transformation $\eta (\lambda, s)=\lambda^{-\theta} s$ appearing in the `shift invariance' relationship above, namely, we study the generalization of a `multiplicatively translational' $\eta$, 
\[
\eta(\lambda\tilde{\lambda}, s)= \eta(\tilde{\lambda}, \eta(\lambda, s))\,.
\]

Previously in the study of the possible functional forms in Equation \eqref{iverson_law} and its specializations, researchers assumed additional representations for the function $\xi_s$.  This approach stems historically from the original work by \cite{Fechner}, but it also has the purpose of helping to narrow down the possible solutions to the functional equations; without further assumptions about the functions $\xi_s$, $\eta$, and $\gamma$, Equation \eqref{iverson_law} allows for too many solutions for these functions. As examples of studies of (specializations of) Equation \eqref{iverson_law} in conjunction with additional representations for $\xi_s$, we find \cite{Falmagne1994}, who studied the near-miss model $\xi_{s}(\lambda x)=  \lambda^{\phi(s)}\xi_s (x)$ while also assuming a `subtractive representation' \[s=u(\xi_s(x))-w(x);\] \cite{Iverson2006b}, who studied Equation \eqref{iverson_law} in the context of a `Fechnerian representation'  
\[
s=u(\xi_s(x))-u(x);
\]
\cite{FalmagneLundberg2000}, who studied the near-miss model assuming a `gain-control representation' 
\[
s=\frac{u(\xi_s (x))-w(x)}{\sigma(x)};
\]
\cite{IversonPavel1981}, who studied the shift invariance relationship 
\[
\xi_{\lambda^\theta s}(\lambda x) = \lambda\, \xi_s (x)
\]
under a similar gain-control representation; \cite{HsuIverson2016}, who studied the model 
\[
\xi_{s}(\lambda x)=  \lambda^{\phi(s)}\xi_{\eta(\lambda, s)} (x)
\]
under a `a gain-control representation' of the form  
\[s=\frac{u(\xi_s (x))-u(x)}{\sigma(x)}; \] 
and \cite{DobleHsu2020}, who studied Equation \eqref{iverson_law} assuming a subtractive representation.  As will be seen, in the present paper we examine representations such as these (Fechnerian, subtractive, and gain-control), but we also consider Equation \eqref{iverson_law} on its own, especially with assumptions about the function $\eta$.

The present work is arranged as follows.  Following \cite{Falmagne1985}, Section 2 contains the most basic concepts such as `psychometric families', `sensitivity functions', and some of their properties (e.g. `anchored', `balanced', and `parallel families'). Here we also present some representations known in psychophysics and their interpretations (e.g. Fechnerian, subtractive, affine, gain-control). Our results can be found in Section 3, which we divide into two subsections. 
First, we study Equation \eqref{iverson_law}, purely from the point of view of the theory of functional equations. In Remark \ref{form_xi}, we show that if there is at least one $s$ for which the function $\lambda\longmapsto \eta(\lambda, s)$ is invertible, then $\xi$ can be written with the aid of one-variable functions in the form $\xi_{s}(x)= \dfrac{\Phi(f(s)x)}{g(s)}$. As a supplement to this, in Theorem \ref{thm_constant} we prove that if there exists a subset $\tilde{S}\subset S$ such that for all $s\in \tilde{S}$, the mapping 
$
 J\ni \lambda \longmapsto \eta(\lambda, s)
$
is constant, then $\xi$ also has a special form, namely 
$
 \xi_{s}(x)= \kappa(s)x^{\rho(s)}. 
$
The form obtained for the function $\xi$ in Remark \ref{form_xi} proves to be very useful when we assume some kind of representation. In line with this, we first describe those functions $\xi$ that have a certain gain-control representation (see Proposition \ref{prop1}), and then we examine Equation \eqref{iverson_law} again in the case where $\xi$ admits a representation 
\[
 \xi_{s}(x)= u(x)+v(s)
\]
(see Proposition \ref{prop3}). This latter representation is closely related to the so-called   parallel families. In Proposition \ref{prop4}, we provide a complete description of  anchored, balanced and parallel families. {These results are summarized in Part A of the graphical abstract.}

Motivated by \cite{HsuIverson2016}, in the second part of the third section we study \eqref{iverson_law} assuming that the mapping $\eta$ is multiplicatively translational. 
First, we show that multiplicatively translational mappings can be written in the form
\[
 \eta(\lambda, s)= H(\lambda \cdot H^{-1}(s))
\]
if certain mild conditions are met, {see Part B of the graphical abstract.} Using this form, we examine what special form the function  $\gamma$ in \eqref{iverson_law} can then take. Later we turn to the examination of Falmagne's power law, i.e., 
\[
 \xi_{s}(\lambda x)= \lambda^{\phi(s)}\xi_{\eta(\lambda, s)}(x)
\]
which is an important special case of \eqref{iverson_law}. According to Remark \ref{cor1}, 
\[
 \xi_{s}(x)= x^{\phi(s)}F(x\cdot H^{-1}(s))
\]
is satisfied. After that, we show that if the function $\phi$ in the exponent is monotonic, then $\phi$ cannot be strictly monotonic on any subinterval of positive length. 
According to Proposition \ref{prop_subtractive}, the corresponding function $\xi$ then has a subtractive representation. As an application of the results of the subsection, we close the paper with the study of the shift invariance 
\[
 \xi_{s}(\lambda x)= \gamma(\lambda, s)\xi_{\lambda^{\theta} s}(x). 
\]
{These latter results are summarized in Part C of the graphical abstract.}

\section{Preliminaries}

Our mathematical results are applicable to each of the experimental contexts mentioned in the Introduction, but for concreteness and uniformity, we will use terminology and notation from the intensity discrimination context.  (Other contexts, especially the partial masking context of \cite{PavelIverson1981}, are referred to when especially relevant.) To this end, suppose a participant must compare a stimulus of intensity $x$ with one of intensity $y$ (both measured in ratio scale units\footnote{{For applications of (specializations of) Iverson's law of similarity to contexts in which the stimuli are measured in log-intensity scales, see \cite{DobleFalmagne2003} and \cite{Doble_et_al2006} }}) and judge which has the greater sensory impact (i.e., is louder, is heavier, is brighter, etc.)\footnote{{In fact, for the discrimination task, there are some subtle differences regarding whether one is concerned about the probability of `subjective responses' or is concerned about the probability of `correct responses'. We have presented discussion in an earlier paper \cite{HsuDoble2015}; in that paper, we also provided a mechanism to reconcile the two psychometric functions (pages 165-166).}}.
Let $P(x, y)$ be the probability that intensity $y$ is judged greater than
intensity $x$. A simple model for these probabilities is 
\[
\tag{$\mathscr{F}$}
P(x, y)= F(u(y)-u(x))\,,
\]
in which $u$ and $F$ are continuous and strictly increasing functions\footnote{{In this article, we confine our discussion to Fechnerian models that assume strictly-increasing internal scales. Regarding the issue of subliminality, see, for example, \cite{HsuDoble2015} for a brief discussion (page 159).}}. This model expresses that the stimulus intensities $x$ and $y$ are scaled by the sensory mechanism to the values $u(x)$ and $u(y)$, resp., and the probabilities $P(x, y)$ are determined by the differences $u(y)-u(x)$. Equation $(\mathscr{F})$ is called a \emph{Fechnerian representation}, and $u$ is termed to be a \emph{scale}. The model is closely related to the problem of \cite{Fechner}; for more details, see especially \cite{DzhafarovColonius1999} and \cite{Falmagne1985}.

This representation has been extended and generalized in several ways. There may be asymmetries between the stimuli, e.g.~biases based on the order or position of stimulus presentations can occur. In this case, it is more appropriate to consider the following model
\[
\tag{$\mathscr{S}$}
P(x, y)= F(u(y)-w(x)),
\]
that is called a \emph{subtractive representation}.

A still more general representation is 
\[
P(x, y)= F(u(y)+h(y)g(x)) \;\;\; \text{or} \;\;\; P(x, y)= F(u(y)h(x)+g(x)),
\]
which is called in the literature an \emph{affine representation}; see \cite{HsuIversonDoble2010}. Affine representations include so-called \emph{gain-control representations}, one of which we examine in this paper, namely, the representation
\[
\tag{$\mathscr{G}$}
P(x, y)= F \left(\frac{u(y)-u(x)}{\sigma(x)}\right)\,.
\]
This representation allows for the fact that the sensory mechanism may adjust its ``gain'' via the normalizing factor $\sigma(x)$.  See  \cite{HsuIversonDoble2010} for further description.

%
%
%

For all fixed $x$, introducing the function $p_{x}$ as
\[
p_{x}(y)=P(x, y),
\]
a family of functions $\mathcal{F}= \left\{ p_{x}\, \vert x\in I\right\}$ arises. Due to the properties of the probability, this family has (among others) the following properties
\begin{itemize}
\item for all fixed $x$, the range of the function $p_{x}$ is contained in the interval $]0, 1[$ ;
\item for all fixed $x$, the function $p_{x}$ is strictly increasing and continuous.
\end{itemize}

In what follows, our goal is to determine such families of functions under certain conditions. Therefore, following \cite{Falmagne1985}, we introduce the following notion.

A family of functions $\mathcal{F}= \left\{ p_{x}\, \vert x\in I\right\}$
is called a \emph{psychometric family}, if
\begin{enumerate}[(i)]
\item for all fixed $x$, the domain of the function $p_{x}$ is an open interval\footnote{In this work every interval has positive length.\label{fn1}} ;
\item for all fixed $x$, the range of the function $p_{x}$ is contained in the interval $]0, 1[$ ;
\item for all fixed $x$, the function $p_{x}$ is strictly increasing and continuous.
\end{enumerate}

As we will see, instead of psychometric families, it will be much more convenient to work with sensitivity functions.

Let $\mathcal{F}= \left\{ p_{a}\, \vert \, a\in I\right\}$ be a psychometric family. The \emph{sensitivity function} of $\mathcal{F}$ is a function $\xi$ defined for all backgrounds (or standards) $a$, and all probabilities $\pi$ in the range of a psychometric function $p_a$, by the equation
\begin{equation}\label{Eq_sensitivity}
\xi_{\pi}(a)=p_{a}^{-1}(\pi),
\end{equation}
in other words,
\[
p_{a}\left(\xi_{\pi}(a)\right)= \pi.
\]
In words: $\xi_{\pi}(a)$ is the intensity of the stimulus yielding a response probability $\pi$, for the background $a$. 
As J.-C.\ Falmagne in \cite{Falmagne1985} writes: \emph{``The change of notation, from $p_{a}^{-1}(\pi)$ to $\xi_{\pi}(a)$, symbolizes an important shift of focus in our analysis. The quantity $\pi$, the probability of the response, ceases to be the variable of interest and becomes the parameter.
Typically, at most a couple of values of $\pi$ are considered in experimental plots of Weber functions or sensitivity functions. By contrast, the effect on $\xi_{\pi}(a)$ of the variable $a$ is investigated in minute detail. This is in line with a long tradition in psychophysical research in which the sensory scales uncovered by the analysis of the data are deemed of central importance.''}

In terms of the sensitivity functions, $(\mathscr{F})$, $(\mathscr{S})$, and $(\mathscr{G})$ resp. can be re-formulated as
\[
s=u(\xi_{s}(x))-u(x), 
\]

\[
s=u(\xi_{s}(x))-w(x),
\]
and 
\[
 s\sigma(x)=u(\xi_{s}(x))-u(x).
\]
We note that in each of these re-formulations, the switch from $\pi$ to $s$ occurs through the process described in Footnote \ref{footnote}.  We will use the notation with $s$ in the sequel.

If we do not assume anything about the sensitivity functions, then an infinite number of scale functions can be considered. In this paper we consider in addition a  \emph{similarity law} initiated by G.~J.~Iverson in \cite{Iverson2006b}. In this case, we assume that the sensitivities fulfill identity  
\[
\xi_{s}(\lambda x)= \gamma(\lambda, s)\xi_{\eta(\lambda, s)}(x)
\]
for all possible values of the variables $\lambda, x$ and $s$ and with appropriate two-variable functions $\gamma$ and $\eta$.

In accordance with \cite{Falmagne1985}, a psychometric family $\mathcal{F}= \left\{ p_{a}\, \vert \, a\in I\right\}$ is called \emph{anchored} if there exists a number $\alpha \in ]0, 1[$ such that
\begin{enumerate}[(i)]
\item for all $a\in I$, there exists an $x$ satisfying $p_{a}(x)= \alpha$;
\item for all $x$, there exists an $a\in I$ such that $p_{a}(x)= \alpha$.
\end{enumerate}
The above conditions mean in words that for every background $a$ there is a stimulus $x$, and for every stimulus $x$ there is a background $a$, such that $p_{a}(x)= \alpha$. A number $\alpha \in ]0, 1[$ satisfying these conditions will be called an \emph{anchor} of $\mathcal{F}$.

A psychometric family $\mathcal{F}$ is called \emph{parallel} if for all $p_{a}, p_{b}\in \mathcal{F}$, we have
\[
p_{a}^{-1}(\pi)-p_{a}^{-1}(\pi^{\ast})=
p_{b}^{-1}(\pi)-p_{b}^{-1}(\pi^{\ast})
\]
whenever all four terms are defined.

If a parallel psychometric family $\mathcal{F}= \left\{ p_{a}\, \vert \, a\in I\right\}$ is given, then and only then
\[
\xi_{\pi}(a)-\xi_{\pi^{\ast}}(a)= \xi_{\pi}(b)-\xi_{\pi^{\ast}}(b)
\]
is fulfilled by the sensitivity functions. In view of \cite[Theorem 8.5]{Falmagne1985}, an
anchored psychometric family $\mathcal{F}= \left\{p_{a}\, \vert \, a\in I \right\}$ is parallel if and only if there exist strictly increasing and continuous functions $u$ and $v$ such that
\[
\xi_{s}(x)= u(x)+v(s)
\]
for all possible $x$ and $s$.

Intuitively, a psychometric family is parallel if any two psychometric functions can be made to coincide by a horizontal shift of one toward the other. This suggests that given one psychometric function, say $p_a$, any other psychometric function $p_b$ is completely characterized by the value of one parameter depending on $b$ that is denoted by $g(b)$, expressing the length and direction of the above-mentioned shift ($g(b)$ may be negative).

A psychometric family $\mathcal{F}= \left\{ p_{a}\, \vert \, a\in I\right\}$ is called \emph{balanced} if
\[
p_{a}(b)+p_{b}(a)=1
\]
holds for all possible $a$ and $b$.

The assumption that the psychometric family $\mathcal{F}= \left\{ p_{a}\, \vert \, a\in I\right\}$ is balanced holds if and only if for the sensitivity function we have
\[
\xi_{1-\pi}\left(\xi_{\pi}(a)\right)=a
\]
for all $\pi$ and $a$.

As J.-C.~Falmagne in \cite{Falmagne1985} writes, there are cases (i.e.~the brightness discrimination experiment) where the balance condition naturally holds by the paradigm or of the method of data collection. At the same time, observe that the order of the stimuli in the notation $(a,b)$ suggests the fact that
stimulus $a$ is presented first, followed by stimulus $b$. The balance condition states,
in effect, that the order of presentation does not affect the result of a comparison.
This may not be true, for example, for auditory stimuli presented successively.

These properties may seem natural due to the aforementioned, but we will see later that these notions are far too restrictive, even without Iverson's law of similarity, see Propositions \ref{prop3} and \ref{prop4}.

To prove Proposition \ref{prop1}, the lemma below will be utilized from \cite{Aczel2005}. {Recall that we say that a function is \emph{philandering} if it is not constant on any interval of positive length. }

\begin{lem}\label{lemma_Aczel}
Let {$D\subset ]0, +\infty[^{2}$} be a {nonempty,} open and connected set. Assume that the real-valued functions $f, g, k, h$ fulfill
\[
f(st)= g(s)+h(s)\cdot k(t)
\]
for all $(s, t)\in D$. If the function $f$ is measurable and philandering, then  
\begin{enumerate}[(A)]
\item either
\[
\begin{array}{rcl}
f(r)&=&\beta_{0}\ln(r)+\beta_{1}+\beta_{2}\beta_{3}\\
g(s)&=&\beta_{0}\ln(s)+\beta_{1}\\
h(s)&=&\beta_{3}\\
k(t)&=&\frac{\beta_{0}}{\beta_{3}}\ln(t)+\beta_{2}
\end{array}
\]
\item or
\[
\begin{array}{rcl}
f(r)&=&\beta_{3}\beta_{4}r^{\beta_{0}}+\beta_{1}\\
g(s)&=&\beta_{2}s^{\beta_{0}}+\beta_{1}\\
h(s)&=&\beta_{3}s^{\beta_{0}}\\
k(t)&=&\beta_{4}t^{\beta_{0}}-\frac{\beta_{2}}{\beta_{3}}
\end{array}
\]
\end{enumerate}
for all $s\in \left\{ s\, \vert \, \text {there exists $t$ such that $(s, t)\in D$} \right\}$, \\
$t\in \left\{ t\, \vert \, \text {there exists $s$ such  that $(s, t)\in D$} \right\}$ and
$r\in \left\{ st \, \vert \, (s, t)\in D\right\}$, where $\beta_{i}$ are real constants for $i=1, 2, 3, 4$ such that $\beta_{0}, \beta_{3}, \beta_{4}\neq 0$.
\end{lem}

Further, in the next section we also use a result of \cite{Lundberg1977} on the functional equation 
\[
 f(\ell(y)+g(s))= m(y)+ h(y+s) 
\]
which is the statement below.

\begin{lem}\label{lem_Lundberg}
Suppose the equation 
\begin{equation}\label{lund_thm_eq}
f(\ell(x) +g(y))=m(x)+h(x+y)
\end{equation}
holds for all $(x, y)$ in a body in $\mathbb{R}^2$, where $f$, $g$, $h$, $\ell$ and $m$ are real-valued, continuous functions, each defined on an interval (possibly different intervals for different functions).  Suppose also that $h$, $\ell$ and $m$  are philandering.  Then the solutions are exactly those given in Cases I-V below, where the parameters can take any real values as long as the resulting expressions are real numbers.  
\begin{enumerate}[{Case }I:]
\item 
\begin{align*}
f(x)&=\alpha + \rho x\\
g(x)&=\beta + b x\\
h(x)&=-\tau + \rho b x\\
\ell(x)& \text{ is arbitrary } \\
m(x)&=\rho \ell(x)- \rho b x+\alpha+\rho\beta+\tau
\end{align*}
\item 
\begin{align*}
f(x)&=\alpha + \rho \log (c+e^{\kappa x})\\
g(x)&= \frac{1}{\kappa}\log(-\beta c + d e^{\delta x})\\
h(x)&= -\tau + \alpha + \rho \log(b c + d e^{\delta x})\\
\ell(x) &=-\frac{1}{\kappa}\log(\beta + b e^{-\delta x})\\
m(x)&=\tau - \rho \log(b + \beta e^{\delta x})
\end{align*}
\item 
\begin{align*}
f(x)&= \rho \log (\alpha - be^{\kappa x})\\
g(x)&= \frac{1}{\kappa}\log(\beta - d \alpha x)\\
h(x)&=  -\tau +  \rho \log(b d \alpha x + \alpha \epsilon - b \beta)\\
\ell(x)& = -\frac{1}{\kappa}\log(\epsilon + b d x)\\
m(x)&= \tau - \rho \log(\epsilon + b d x)
\end{align*}
\item 
\begin{align*}
f(x)&= \alpha + \rho e^{\kappa x}\\
g(x)&= \beta + \frac{1}{\kappa}\log(b + c e^{\delta x})\\
h(x)&=  -\tau + \alpha + \rho c e^{\delta x}\\
\ell(x) &=-\beta + \frac{\delta}{\kappa}x\\
m(x)&= \tau + \rho b e^{\delta x}
\end{align*}
\item 
\begin{align*}
f(x)&= \alpha + \frac{\rho}{\delta} \log (\beta + x)\\
g(x)&= -\beta - \epsilon + ce^{\delta x}\\
h(x)&=  -\tau + \alpha + \frac{\rho}{\delta} \log (b+c e^{\delta x})\\
\ell(x) &=  \epsilon + b e^{-\delta x}\\
m(x)&= \tau - \rho x
\end{align*}
\end{enumerate}
\end{lem}

We end this section with some short comments.  Recall from Footnote \ref{fn1} that all intervals in this paper have positive length.  Another is that in several results below, we assume that the value $s=0$ is attainable.  Note that $s=0$ is feasible in all of the experimental contexts mentioned in the Introduction. In the case of intensity discrimination experiments, the index $s=0$ may be assumed to correspond to the probability of $1/2$ (or to the $d'$ value of $0$), and for other experiments such as partial masking, the intensity $s=0$ would also make sense for the model.  The third  short comment is that even though we are using terminology adopted from an intensity discrimination context, we will continue to use that terminology when referring to other experimental contexts.  For example, we use the term `sensitivity function' even in a partial masking context.

\section{Results}

\subsection{Iverson's law of similarity}

According to our primary aim, now we study Iverson's law of similarity, i.e., identity
\begin{equation}\label{Eq_Iverson}
\xi_{s}(\lambda x)= \gamma(\lambda , s)\xi_{\eta(\lambda, s)}(x)
\qquad
\left(x\in I, \lambda\in J, s\in S\right). 
\end{equation}
In the following statement, we show that if there exists an $s^{\ast}\in S$ such that the mapping $\lambda \longmapsto \eta(\lambda, s^{\ast})$ is invertible, then the  function $\xi$ can be expressed with the aid of one-variable functions in a rather special form.

Before this, however, we would like to clarify under which conditions we will be able to prove this statement. Additionally, we would like to compare our assumptions here with the standard assumptions found in the literature. In what follows (until stated otherwise), $\xi$ will always denote a one-parameter family that is defined on the Cartesian product $I\times S$ of some real intervals $I$ and $S$. This means that we do not assume that $\xi$ is a sensitivity function, which among other things would imply that for all fixed $s$, the mapping $\xi_{s}$ is continuous and strictly increasing. Further, the domain of the two-variable function $\gamma$ and $\eta$ is supposed to be $J\times S$. In the case of the function $\gamma$, there is no assumption on the range, while in the case of the function $\eta$ we assume that it maps $J\times S$ into $S$. This latter assumption is justified by Equation \eqref{Eq_Iverson} itself. It is important to emphasize that apart from these, we do not impose any a priori regularity assumptions on the involved functions, but in several cases, we get them as a consequence. Finally, the assumption $I\cdot J\subset I$ helps Equation \eqref{Eq_Iverson} to hold for all $x\in I$, $\lambda\in J$ and $s\in S$. Thus, from one point of view, in this paper we will work under more general conditions, since in most papers the authors assume that the functions $\gamma$ and $\eta$ are continuous in each of their variables. On the other hand, the domain of $\xi$ is assumed to be $I\times S$, which, compared to other works, is a bit more special.  As an example, we mention that in \cite{DobleHsu2020}, $s$ is taken from an open interval $S$; for all $s\in S$, the mapping $\xi_{s}$ is defined on an open interval $I_{s}\subset ]0, +\infty[$ such that the set
\[
B= \left\{(x, s)\, \vert \, s\in S, x\in I_{s}\right\} \subset \mathbb{R}^{2}
\]
is an open and connected set and Equation \eqref{Eq_Iverson} is supposed to hold for all triples $(x, \lambda, s)$ such that $\lambda \in J$, $(x, s)\in B$ and $(\lambda x, s)\in B$.

\begin{rem}\label{form_xi}
Let $I, J, S\subset \mathbb{R}$ be  intervals with $I\cdot J= \left\{ \lambda x\, \vert \, x\in I, \lambda \in J \right\}\subset I$. Let further $\gamma\colon J\times S\to \mathbb{R}$ and $\eta \colon J\times S\to S$ be functions and assume that the one-parameter family of functions $\xi_{s}\colon I\to \mathbb{R}\, (s\in S)$ fulfills
\[
\xi_{s}(\lambda x)= \gamma(\lambda , s)\xi_{\eta(\lambda, s)}(x)
\qquad
\left(x\in I, \lambda\in J, s\in S\right).
\]
If there exists an $s^{\ast}\in S$ such that the mapping $\lambda \longmapsto \eta(\lambda, s^{\ast})$ is invertible, then there exist functions $f\colon \tilde{S}\to J$, $g\colon \tilde{S}\to \mathbb{R}$ and $\Phi \colon I\to \mathbb{R}$ such that
\[
g(s) \cdot \xi_{s}(x)= \Phi(f(s)\cdot x)
\qquad
\left(x\in I, s\in \tilde{S}\right). 
\]
Here $\tilde{S}$ denotes the range of the mapping $\lambda \longmapsto \eta(\lambda, s^{\ast})$, that is, $\tilde{S}= \eta(J, s^{\ast})$. 
Indeed, by our assumptions, there exists an $s^{\ast}\in J$ such that the mapping $\lambda \longmapsto \eta(\lambda, s^{\ast})$ is invertible. Then the inverse of this mapping is defined on the range of $\lambda \longmapsto \eta(\lambda, s^{\ast})$, that is, on $\tilde{S}= \eta(J, s^{\ast})$. Therefore, the mapping 
\[
 J\ni \lambda \longmapsto \eta(\lambda, s^{\ast})\in \tilde{S}
\]
is a \emph{bijection} between the sets $J$ and $\tilde{S}$. In other words, for all $s\in \tilde{S}$, there exists a uniquely determined $\lambda \in J$ such that $\eta(\lambda, s^{\ast})= s$, or equivalently $\lambda= \eta^{-1}(s, s^{\ast})$. 

The above identity with $s^{\ast}$ instead of $s$ is just
\[
\tag{$\ast$}
\xi_{s^{\ast}}(\lambda x)= \gamma(\lambda , s^{\ast})\xi_{\eta(\lambda, s^{\ast})}(x)
\qquad
\left(x\in I, \lambda \in J\right).
\]
Let now $s\in \tilde{S}$ be arbitrary. 
From this, with the substitution $s= \eta^{-1}(\lambda, s^{\ast})$ we get that
\[
\xi_{s^{\ast}}(\eta^{-1}(s, s^{\ast}) x)= \gamma(\eta^{-1}(s, s^{\ast}) , s^{\ast})\xi_{s}(x)
\qquad
\left(x\in I, s \in \tilde{S}\right).
\]

Let now
\[
f(s)=\eta^{-1}(s, s^{\ast})
\qquad
g(s)= \gamma(\eta^{-1}(s, s^{\ast}) , s^{\ast})
\]
and
\[
\Phi(x)= \xi_{s^{*}}(x)
\qquad
\left(x\in I, s\in \tilde{S}\right)
\]
to deduce the statement.
\end{rem}

\begin{rem}
Using the notations and the assumptions of Remark \ref{form_xi}, if the one-parameter family of functions $\xi_{s}\colon I\to \mathbb{R}\, (s\in J)$ fulfills
\[
\xi_{s}(\lambda x)= \gamma(\lambda , s)\xi_{\eta(\lambda, s)}(x)
\qquad
\left(x\in I, \lambda\in J\setminus \left\{0\right\}, s\in \tilde{S}\right),
\]
then $\gamma(\lambda^{\ast}, s)= 0$ implies that $\xi_{s}(x)=0$ on some subinterval of $I$.  
Indeed, let $s\in \tilde{S}$ be fixed and assume that there exists $\lambda^{\ast}\in J\setminus \left\{0\right\}$ such that $\gamma(\lambda^{\ast}, s)=0$, then
\[
\xi_{s}(\lambda^{\ast} x)= \gamma(\lambda^{\ast}, s)\xi_{\eta(\lambda^{\ast}, s)}\left(x\right)=0
\]
for all $x\in I$, yielding that $\xi_{s}$ is identically zero on the interval $\lambda^{\ast}I= \left\{\lambda^{\ast}x \, \vert \, x\in I \right\}$.

In some contexts, it is assumed that $\xi_{s}(x)>0$ 
for all possible $s$ and all possible $x$.
Note that in such a case, it follows from Remark \ref{form_xi} that we have
\[
\xi_{s}(x)= \frac{\Phi(f(s)x)}{g(s)}
\]
for all $x\in I$ and $s\in \tilde{S}$.
\end{rem}

\begin{rem}
The assumption that there exists $s^{\ast}\in S$ such that the mapping
\[
J\ni \lambda \longmapsto \eta(\lambda, s^{\ast})
\]
is invertible, is essential while proving Remark \ref{form_xi}. However, from the point of view of applications, this condition may seem artificial. Instead of invertibility, for example, monotonicity is a much more reasonable assumption. Note, however, that then the following alternative is fulfilled. Accordingly, suppose that there exists an $s^{\ast}\in S$ such that the mapping $\lambda \longmapsto \eta(\lambda, s^{\ast})$ is monotonic on the interval $J$. Then the points of discontinuity of this function form a countable set. These points of discontinuity induce a (countable) disjoint partition $(J_{\alpha})_{\alpha \in A}$ of the interval $J$. Then for all $\alpha \in A$ the mapping
\[
J^{\circ}_{\alpha}\ni \lambda \longmapsto \eta(\lambda, s^{\ast})
\]
is continuous (and monotonic), where $J^{\circ}_{\alpha}$ denotes the interior of the interval $J_{\alpha}$. Then there are subintervals $J_{\alpha, c}, J_{\alpha, s}\subset J^{\circ}_{\alpha}$ such that the above mapping is constant on $J_{\alpha, c}$ and strictly monotonic on $J_{\alpha, s}$. The mapping
\[
\lambda \longmapsto \eta(\lambda, s^{\ast})
\]
while restricted to $J_{\alpha, s}$ is invertible. Thus the assumptions of Remark \ref{form_xi} are fulfilled on $J_{\alpha, s}$.
 Further, the same mapping, while restricted to $J_{\alpha, c}$, is a constant mapping.  
 This motivates the following theorem, in which Iverson's law of similarity is examined in the case where the mapping $\eta$ is constant for certain values $s$. 
 Notice that in this statement, in contrast to Remark \ref{form_xi}, we only need to require the minimum about the domains of the involved functions.
\end{rem}

{
\begin{thm}\label{thm_constant}
 Let $S$ and $J$ be intervals and suppose that for all $s\in S$, we are given an interval $I_{s}\subset ]0, +\infty[$ and let 
 \[
  B= \left\{(x, s)\, \vert \, s\in S, x\in I_{s}\right\}\subset \mathbb{R}^{2}. 
 \]
Suppose further that 
\begin{enumerate}[(i)]
\item there exists $s^{\ast}\in S$  such that the mapping 
\[
 J\ni \lambda \longmapsto \eta(\lambda, s^*)
\]
is constant, while the mapping 
\[
 I_{s^{*}}\ni x\longmapsto \xi_{s^*}(x)
\]
is philandering and measurable;
\item we have 
{
\[
 \xi_{s^{*}}(\lambda x)= \gamma(\lambda, s^{*})\xi_{\eta(\lambda, s^{*})}(x)
\]
}
for all $\lambda\in J$ and $(x, s^*)\in B$ for which also $(\lambda x, s^*)\in B$;
\item the set 
\[
\left\{ (\lambda, x) \, \vert \, \lambda\in J, (x, s^*)\in B  \text{ for which also } (\lambda x, s^*)\in B\right\} \subset ]0, +\infty[^{2}
\]
is open and connected. 
\end{enumerate}
Then there exist constants  $\kappa(s^*)$ and $\rho(s^*)$ such that 
\[
 \xi_{s^*}(x)= \kappa(s^*)x^{\rho(s^*)} 
 \qquad 
 \left(x\in I_{s^*}\right). 
\]
\end{thm}
}
\begin{proof}
 Since the mapping 
 \[
 J\ni \lambda \longmapsto \eta(\lambda, s^{\ast})
\]
is constant, Iverson's law of similarity with $s=s^{\ast}$, i.e. 
{
\[
 \xi_{s^{\ast}}(\lambda x)= \gamma(\lambda, s^{\ast})\xi_{\eta(\lambda, s^{\ast})}(x)
 \qquad 
 \left(x\in I_{s*}, \lambda \in J\right), 
\]
}
becomes a multiplicative Pexider equation. This equation is satisfied on the open and connected set 
\[
\left\{ (\lambda, x) \, \vert \, \lambda\in J, (x, s^*)\in B  \text{ for which also } (\lambda x, s^*)\in B\right\} \subset ]0, +\infty[^{2}
.\] 

Further, the function $\xi_{s^{\ast}}$ is philandering and  measurable on the interval $I_{s^{\ast}}$. Thus, from Lemma \ref{lemma_Aczel} (or directly from \cite[Corollary 6]{Aczel2005}) we obtain that  there exist real constants (depending on $s^{\ast}$), say $\kappa(s^{\ast})$ and $\rho(s^{\ast})$ such that 
{
\begin{equation}\label{eq_constant}
 \xi_{s^{\ast}}(x)= \kappa(s^{\ast})x^{\rho(s^{\ast})} 
 \qquad 
 \left(x\in I_{s^{\ast}}\right). 
\end{equation}
}
\end{proof}

\begin{rem}
 It can be seen from the steps of the above proof that we do not necessarily have to assume that the function $\xi_{s}$ is measurable for all $s\in \tilde{S}$. However, then we only get a `local' version of the above statement. More precisely, if we assume that there is an $s^{\ast}$ for which $\xi_{s^{\ast}}$ is measurable, then we get  that there exist real constants $\kappa(s^{\ast})$ and $\rho(s^{\ast})$ such that 
{representation \eqref{eq_constant} holds for all $x\in I_{s^{\ast}}$.} 
\end{rem}

\begin{rem}
 Notice that the representation obtained in the above theorem satisfies an identity much more specific than Iverson's law of similarity. Indeed, if 
 \[
  \xi_{s}(x)= \kappa(s)x^{\rho(s)}
\]
is fulfilled for all $s\in \tilde{S}$ and $x\in I_{s}$, then 
{
\[
 \xi_{s}(\lambda x) = \lambda^{\rho(s)} \xi_{s}(x) 
\]
}
for all $\lambda\in J$ and $(x, s)\in B$ for which also $(\lambda x, s)\in B$, that is, Falmagne's power law model holds. 
As mentioned earlier, in the case of intensity discrimination, this model is called the near-miss to Weber's law when the exponent $\rho(s)$ is close to but different from $1$ for some or all $s$, and it  gives Weber's law when  $\rho(s)=1$ for all $s$. {This law (in the form $\xi_{s}(x)= \kappa(s)x^{\rho(s)}$ or $\xi_{s}(\lambda x) = \lambda^{\rho(s)} \xi_{s}(x)$) is studied in some detail in \cite{Doble_et_al2006} and \cite{Augustin2008}. In the latter, it is shown that the law is compatible with common psychometric functions such as the Weibull, logistic, Gaussian, and Gumbel distributions, with possible restrictions on the functions $\kappa(s)$ and $\rho(s)$.  But note that this law is not general enough to cover some empirical situations, such as the partial masking experiment mentioned in the Introduction, for which the shift invariance model $\xi_{\lambda^\theta s}(\lambda x) = \lambda\, \xi_s (x)$ has been shown to be appropriate; see \cite{IversonPavel1981}.  (This shift invariance model is distinct from the model $\xi_{s}(\lambda x) = \lambda^{\rho(s)} \xi_{s}(x)$, as can be seen by the presence of $\xi_{\lambda^\theta s}$ in the shift invariance model, which is indexed by $\lambda^\theta s$ rather than by $s$.)}
\end{rem}

So we see from Remark \ref{form_xi} and Theorem \ref{thm_constant} that assumptions about $\eta$ give helpful information about Iverson's law of similarity without assuming a further representation for $\xi$.   Theorem \ref{thm_constant} says that for $s$ for which the mapping $\lambda \longmapsto \eta(\lambda, s)$ is constant, Iverson's law of similarity simplifies to Falmagne's power law model.  In the extreme case that this mapping is constant for all $s$ (that is, $\eta(\lambda, s)=\nu(s)$ for some function $\nu$, for all $(\lambda, s)$), Falmagne's power law model holds for all $s$.  Theorem \ref{form_xi} looks at a case for which there is some $s^{\ast}$ for which this does not hold, and more specifically, that there is some $s^{\ast}$ such that $\lambda \longmapsto \eta(\lambda, s^{\ast})$ is invertible.  In such a case, Theorem \ref{form_xi} shows that the function $\xi$ can be expressed with the aid of one-variable functions. 

In the following, we determine those mappings which have the form that appears in Remark \ref{form_xi} and that also have a gain-control type affine representation of the form
\[
\xi_{s}(x)= u^{-1}\left(s\sigma(x)+u(x)\right)
\]
for all possible $x$ and $s$, and with appropriate functions $u$ and $\sigma$.

\begin{prop}\label{prop1}
Let {$I, J, S\subset [0, +\infty[$} be  intervals and suppose that the one-parameter family of functions $\xi_{s}\colon I\to {[0, +\infty[}\, (s\in S)$  admits a representation 
\[
\xi_{s}(x)= \frac{\Phi(f(s)x)}{g(s)}
\qquad
\left(x\in I, s\in S\right)
\]
with some functions $\Phi, f$ and {nowhere zero function $g$}.
Assume further, that {the family $\left\{ \xi_{s}\, \vert \, s\in S\right\}$} fulfills the following affine representation
\[
\xi_{s}(x)= u^{-1}\left(s\sigma(x)+u(x)\right)
\]
for all possible $x\in I$ and $s\in S$ with an appropriate strictly monotonic {and continuous} function $u$ and with an appropriate function $\sigma$. 
If $0\in S$, then
{
\begin{enumerate}[(A)]
\item either there exist nonzero real numbers $c$ and $\mu$ and an appropriate real number $d$ such that 
\[
\xi_{s}(x)= \sqrt[\mu]{\frac{s+d}{c}}  \cdot x
\qquad
\left(x\in I, s\in S\right), 
\]
\item or there exist real numbers $c, d$ with $c\neq 0$ such that 
\[
\xi_{s}(x)= \exp\left(\frac{s-d}{c}\right)  \cdot x
\qquad
\left(x\in I, s\in S\right), 
\]
\item or we have 
\[
 \xi_{s}(x)=x 
 \qquad 
 \left(x\in I, s\in S\right). 
\]
\end{enumerate}}
\end{prop}

\begin{proof}
Under the assumptions of the proposition, we have
\[
\frac{\Phi(f(s)x)}{g(s)}= u^{-1}\left(s\sigma(x)+u(x)\right)
\qquad
\left(x\in I, s\in S\right).
\]
Let $s=0$ in this identity to deduce that
\[
\frac{\Phi(f(0)x)}{g(0)}= u^{-1}(u(x))= x
\qquad
\left(x\in I\right),
\]
that is,
\[
\Phi(x)= \alpha x
\qquad
\left(x\in I\right),
\]
where $\alpha= \dfrac{g(0)}{f(0)}$. Note that this already yields that
\[
\xi_{s}(x)= \alpha(s)x
\qquad
\left(x\in I, s\in S\right)
\]
with an appropriate function $\alpha \colon S\to \mathbb{R}$.

Therefore 
\[
 \alpha(s)x= u^{-1}(s\sigma(x)+u(x)) 
 \qquad 
 \left(x\in I, s\in S\right), 
\]
or equivalently 
\[
 u(\alpha(s)x)= s\sigma(x)+u(x)
 \qquad 
 \left(x\in I, s\in S\right). 
\]
If $\sigma$ is the identically zero function, then we obtain 
\[
 \alpha(s)x=x 
 \qquad 
 \left(x\in I, s\in S\right), 
\]
from which $\alpha(s)=1$ follows for all $s\in S$. This leads to alternative (C). 

If $\sigma$ is not the identically zero function, then there exists an $x^{\ast}\in I$ such that $\sigma(x^{\ast})\neq 0$. Note that if such a point $x^{\ast}$ exists, then we necessarily have $x^{\ast}\neq 0$. Indeed, if we would have $x^{\ast}=0$, then the above identity would take the following form 
\[
 u(0)= s\sigma(0)+u(0) 
 \qquad 
 \left(s\in S\right), 
\]
which is impossible. 
Further, the above identity with $x^{\ast}$ instead of $x$ is 
\[
  \alpha(s)x^{\ast}= u^{-1}(s\sigma(x^{\ast})+u(x^{\ast})) 
 \qquad 
 \left(s\in S\right). 
\]
Since the function $u$ is strictly monotonic and continuous, the right hand side, as a function of the variable $s$, i.e., the mapping 
\[
 s\longmapsto u^{-1}(s\sigma(x^{\ast})+u(x^{\ast})) 
\]
is also strictly monotonic and continuous. From this we obtain that the function $\alpha$ is strictly monotonic and continuous. Thus the set $\alpha(S)= \left\{\alpha(s)\, \vert \, s\in S\right\} {\subset [0, +\infty[}$ is an interval, and the function $\alpha$ is invertible. 
Therefore we have 
\[
 u(t x)= \alpha^{-1}(t)\sigma(x)+u(x)
 \qquad 
 \left(x\in I{\setminus \left\{ 0\right\}}, t \in \alpha(S){\setminus \left\{ 0\right\}}\right).
\]


In view of Lemma \ref{lemma_Aczel}\footnote{To apply Lemma \ref{lemma_Aczel}, the intervals $I{\setminus \left\{ 0\right\}}$ and $\alpha(J){\setminus \left\{ 0\right\}}$ should be open intervals of $]0, +\infty[$. This may not be fulfilled. At the same time, if this is the case, then we work on the interiors of these intervals. The theorem of Aczél is then applicable to these open intervals. Note, however, that when we move to the interiors of the mentioned intervals, we "lose" at most two points, the endpoints. However, since the functions in the statement are continuous, the function values at the points in question can be obtained with taking the limits at the end of the proof.}, the function $\alpha^{-1}$ and thus $\xi$ can be determined. {Either there exist real constants $c, d$ with $c\neq 0$ such that 
\[
\alpha^{-1}(t)= c\ln(t)+d 
\qquad 
\left(t\in \alpha(S) \right), 
\]
that is, 
\[
\alpha(s)= \exp\left(\frac{s-d}{c}\right) 
\qquad 
\left(s\in S\right)
\]
and then 
\[
\xi_{s}(x)= \exp\left(\frac{s-d}{c}\right)  \cdot x
\qquad
\left(x\in I, s\in S\right),
\]
or we have
\[
\alpha^{-1}(t)= c t^{\mu}-d  
\qquad 
\left(t\in \alpha(S)\right)
\]
with some nonzero real numbers $c$ and $\mu$ and with an appropriate real number $d$. 
Then 
\[
\alpha(s)= \sqrt[\mu]{\frac{s+d}{c}} 
\qquad 
\left(s\in S\right)
\]
so we have 
\[
\xi_{s}(x)= \sqrt[\mu]{\frac{s+d}{c}}  \cdot x
\qquad
\left(x\in I, s\in S\right).
\]}
\end{proof}

{
\begin{rem}
Notice that from the above proof we can also obtain the possible forms of the functions $u$ and $\sigma$. Indeed, either
 \[
 u(x)= c_{1} c x^{\mu}+d_{1}
 \qquad
 \text{and}
 \qquad
 \sigma(x)= c_{1} x^{\mu}
 \qquad
 \left(x\in I\setminus \left\{ 0\right\}\right)
 \]
 with some nonzero real numbers $c_{1}, \mu$ and with an appropriate real number $d_{1}$ (in the above formula, $c$ is the same constant as it appears in alternative (A)), 
 or there exists a nonzero real constant $\sigma$  such that
 \[
 u(x)= \sigma (c \ln(x)+d)
 \qquad
 \text{and}
 \qquad
 \sigma(x)= \sigma
 \qquad
 \left(x\in I\setminus \left\{ 0\right\}\right), 
 \]
 where the constants $c, d$ are the same as in alternative (B). 
In the case of alternative (C), the function $\sigma$ is the identically zero function and the function $u$ can be any strictly monotonic and continuous function. 
\end{rem}
}

\begin{rem}
 If the conditions of the above proposition are fulfilled, then we deduce that 
in all of alternatives (A)--(C), the mapping $\xi_s$ must satisfy Weber's law, i.e. $\xi_s (\lambda x) = \lambda \xi_s (x)$ must hold for all possible values of $x$ and $s$.  As Weber's law is too restrictive for many experimental situations, we conclude that the assumptions of the proposition (the forms $\xi_{s}(x)= \frac{\Phi(f(s)x)}{g(s)}$ and $\xi_{s}(x)= u^{-1}\left(s\sigma(x)+u(x)\right)$, and the fact that $0 \in S$) are quite limiting.  (Especially, we see that Alternative (C), for which $\xi_s$ does not depend on $s$, would not occur in any applications under consideration.)  However, see Remark \ref{rem3} and the representation considered in Proposition \ref{prop3} below. 
\end{rem}

\begin{rem}\label{rem3}
Let $I, S= ]0, +\infty[$ and consider the sensitivity functions $\xi_{s}$ defined on $I$ for all $s\in S$ by
\[
\xi_{s}(x)= x+s
\qquad
\left(x\in I, s\in S\right).
\]
These functions fulfill Iverson's law of similarity since we have
\[
\xi_{s}(\lambda x)= s+\lambda x= \lambda \left( \frac{s}{\lambda}+x\right)=
\lambda \xi_{\frac{s}{\lambda}}(x)
\]
for all $\lambda, x\in I$ and $s\in S$, that is, in this case
\[
\gamma(\lambda, s)= \lambda
\qquad
\text{and}
\qquad
\eta(\lambda, s) = \frac{s}{\lambda}
\qquad
\left(\lambda\in I, s\in S\right).
\]
Moreover, in accordance with Remark \ref{form_xi} we have
{
\[
\xi_{s}(x)= \frac{\Phi(f(s)x)}{g(s)}
\qquad
\left(x\in I, s\in S\right),
\]}
since
\[
\xi_{s}(x)= s+x= s \left(1+\frac{x}{s}\right)= \dfrac{\left(1+\frac{x}{s}\right)}{\frac{1}{s}}=
\dfrac{\left(1+\frac{1}{s}x\right)}{\frac{1}{s}}
\qquad
\left(x\in I, s\in S\right),
\]
i.e.,
\[
f(s)= \frac{1}{s}
\qquad
g(s)= \frac{1}{s}
\qquad
\text{and}
\qquad
\Phi(x)=x+1
\qquad
\left(x\in I, s\in S\right).
\]
Finally, we have
\[
\xi_{s}(x)= u^{-1}(s+u(x))
\qquad
\left(x\in I, s\in S\right),
\]
i.e., a gain-control type affine representation with 
\[
u(x)=x
\qquad
\text{and}
\qquad
\sigma(x)=1
\qquad
\left(x\in I\right).
\]
Nevertheless, the sensitivity functions $\xi_s$ cannot be written in any of the forms (A)--(C) in Proposition \ref{prop1}. This is caused by the fact that in this example the assumption $0\in S$ does not hold.  
\end{rem}

We now examine the form $\xi_{s}(x)= \frac{\Phi(f(s)x)}{g(s)}$ along with a different representation than in Proposition \ref{prop1}, namely, we examine the representation $\xi_{s}(x)= u(x)+v(s)$.
Recall that this is the representation associated with an anchored and parallel psychometric family, although here we do not assume that $u$ and $v$ are continuous.

\begin{prop}\label{prop3}
Let $I, S\subset ]0, +\infty[$ be  intervals and suppose that the one-parameter family of functions $\xi_{s}\colon I\to \mathbb{R}\, (s\in J)$  admits a representation
\[
\xi_{s}(x)= \frac{\Phi(f(s)x)}{g(s)}
\qquad
\left(x\in I, s\in S\right)
\]
with some function $\Phi$ and {nowhere zero functions $f$ and $g$}. Suppose further that the mapping $\xi$ admits a representation of the form
\begin{equation}\label{parallel}
\xi_{s}(x)= u(x)+v(s)
\qquad
\left(x\in I, s\in S\right). 
\end{equation}
If the function $\Phi$ is measurable and philandering and $f$ is strictly monotonic, positive and continuous, then  
\begin{enumerate}[(a)]
\item either there exist real constants $\alpha, \beta, \gamma$ with $\alpha\neq 0$ 
such that
\[
u(x)= \alpha \ln(x)+\beta
\qquad
v(s)= \alpha \ln(s)+\gamma
\qquad
\left(x\in I, s\in S\right)
\]
and
\[
\xi_{s}(x)= \alpha \ln(f(s)x)+\beta +\gamma
\qquad
\left(x\in I, s\in S\right)
\]
\item or there exist real constants $\alpha, \beta, \gamma, \rho$ with $\alpha, \beta, \rho \neq 0$ 
such that
\[
u(x)= \alpha x^{\rho}-\beta
\qquad
v(s)= \frac{\gamma}{f(s)^{\rho}}+\beta
\qquad
\left(x\in I, s\in S\right)
\]
and
\[
\xi_{s}(x)= \alpha x^{\rho}+\frac{\gamma}{f(s)^{\rho}}
\qquad
\qquad
\left(x\in I, s\in S\right).
\]
\end{enumerate}
\end{prop}

\begin{proof}
{Combining the assumptions of the proposition, 
\[
\Phi(f(s)\cdot x)= g(s)u(x)+g(s)v(s)
\qquad
\left(x\in I, s\in S\right).
\]}
Since the function $f$ is strictly monotonic and continuous, the set $f(S)$ is an interval and $f$ is invertible. Thus we get that
\[
\Phi(sx)= g(f^{-1}(s))u(x)+g(f^{-1}(s))v(f^{-1}(s))
\]
for all $s\in f(S)$ and $x\in I$. Further, the function $\Phi$ is assumed to be measurable and  philandering, therefore Lemma \ref{lemma_Aczel} applies. This means that  
\begin{enumerate}[(a)]
\item either
\[
\begin{array}{rcl}
\Phi(x)&=&\beta_{0}\ln(x)+\beta_{1}+\beta_{2}\beta_{3}\\
g(f^{-1}(s))v(f^{-1}(s)) &=&\beta_{1}\ln(s)+\beta_{1}\\
g(f^{-1}(s)) &=&\beta_{3}\\
u(x)&=&\frac{\beta_{0}}{\beta_{3}}\ln(x)+\beta_{2}
\end{array}
\]
\item or
\[
\begin{array}{rcl}
\Phi(x)&=&\beta_{3}\beta_{4}x^{\beta_{0}}+\beta_{1}\\
g(f^{-1}(s))v(f^{-1}(s)) &=&\beta_{2}s^{\beta_{0}}+\beta_{1}\\
g(f^{-1}(s))&=&\beta_{3}s^{\beta_{0}}\\
u(x)&=&\beta_{4}x^{\beta_{0}}-\frac{\beta_{2}}{\beta_{3}},
\end{array}
\]
where $\beta_{i}$ are real constants for $i=1, 2, 3, 4$ such that $\beta_{0}, \beta_{3}, \beta_{4}\neq 0$.
\end{enumerate}
Let us consider alternative (a) first. The third identity shows that the function $g$ is constant. Using this in the second identity,
\[
\beta_{3}v(f^{-1}(s))=\beta_{1}\ln(s)+\beta_{1},
\]
i.e.,
\[
v(s) =\frac{\beta_{1}}{\beta_{3}}\ln(f(s))+\frac{\beta_{1}}{\beta_{3}}
\]
follows for all $s\in S$.
So
\[
\xi_{s}(x)= u(x)+v(s)=
\frac{\beta_{0}}{\beta_{3}}\ln(x)+\beta_{2} +
\frac{\beta_{1}}{\beta_{3}}\ln(f(s))+\frac{\beta_{1}}{\beta_{3}}.
\]
At the same time, we also have
\[
g(s)\xi_{s}(x)= \Phi(f(s)x).
\]
Comparing the coefficients, this is possible only if $\beta_{0}= \beta_{1}$.
Therefore
\[
\begin{array}{rcl}
u(x)&=& {\dfrac{{\beta}_{1}\,\ln(x)}{{\beta}_{3}}}+{\beta}_{2}\\[3mm]
v(s)&=& {\dfrac{{\beta}_{1}\,\ln( f\left(s\right))}{{\beta}_{3}}}+{\dfrac{
{\beta}_{1}}{{\beta}_{3}}}
\end{array}
\qquad
\left(x\in I, s\in S\right)
\]
and
\[
\xi_{s}(x)=
{\frac{{\beta}_{1}\,\ln(x)}{{\beta}_{3}}}+{\frac{{\beta}_{1}\,\ln(f \left(s\right))}{{\beta}_{3}}}+{\frac{{\beta}_{1}}{{\beta}_{3}}}+{\beta}_{2}
\]
for all $x\in I$ and $s\in S$.

Finally, consider alternative (b). The third identity yields that
\[
g(s)= \beta_{3}f(s)^{\beta_{0}}
\qquad
\left(s\in S\right)
\]
{and} the second identity implies that
\[
v(x)= {\frac{{\beta}_{2}\,f\left(s\right)^{{\beta}_{0}}+{\beta}_{1}}{
{\beta}_{3}\,f\left(s\right)^{{\beta}_{0}}}}
= \frac{\beta_{2}}{\beta_{3}}+\frac{\beta_{1}}{\beta_{3}} f(s)^{-\beta_{0}}.
\]
Since
\[
u(x)= \beta_{4}x^{\beta_{0}}-\frac{\beta_{2}}{\beta_{3}}
\qquad
\left(x\in I\right),
\]
we obtain that
\[
\xi_{s}(x)=
{\beta}_{4}\,x^{{\beta}_{0}}+{\frac{{\beta}_{2}\,f\left(s\right)^{
{\beta}_{0}}+{\beta}_{1}}{{\beta}_{3}\,f\left(s\right)^{
{\beta}_{0}}}}-{\frac{{\beta}_{2}}{{\beta}_{3}}}
=
{\beta}_{4}\,x^{{\beta}_{0}} +\frac{\beta_{1}}{\beta_{3}} f(s)^{-\beta_{0}}
\qquad
\left(x\in I, s\in S\right).
\]

\end{proof}

\begin{rem}
Observe that in contrast to Proposition \ref{prop1}, in Proposition \ref{prop3} there was no need to assume that $0\in S$. Thus, the sensitivity function
that was considered in Remark \ref{rem3}, appears here as a possible solution. Indeed, let us take {in (b)}
\[
\alpha = \gamma = \rho =1 \quad
\text{and}
\quad
\beta =0,
\]
further
\[
f(s)= \frac{1}{s}
\]
to get that
\[
\xi_{s}(x)= x+s.
\]
\end{rem}

If instead of general one-parameter family of functions, we consider a sensitivity function $\xi$, then identity \eqref{parallel} expresses the fact that $\xi$ is the sensitivity function of an anchored and parallel psychometric family $\mathcal{F}= \left\{ p_{a}\, \vert \, a\in I\right\}$.

As the following statement shows this assumption is quite restrictive even \emph{without} Iverson's law of similarity. Below we describe those sensitivity functions that stem from an anchored, parallel, and balanced psychometric family.

\begin{prop}\label{prop4}
Let $\mathcal{F}= \left\{ p_{a}\, \vert \, a\in I\right\}$ be a psychometric family that is anchored, parallel and balanced, and denote by $\xi$ the sensitivity function of $\mathcal{F}$. Then there exists a function $\nu\colon S\to \mathbb{R}$ which is antisymmetric with respect to the point $s_{0}= \frac{1}{2}$, i.e.
\[
\nu(1-s)= -\nu(s)
\qquad
\left(s\in S\right)
\]
such that
\[
\xi_{s}(x)= x+\nu(s)
\]
holds for all $x\in I$ and $s\in S$.
\end{prop}

\begin{proof}

Due to the anchored and parallel properties of $\mathcal{F}$, there exist strictly increasing, continuous functions $u$ and $v$ such that 
\[
\xi_{s}(x)= u(x)+v(s)
\qquad
\left(x\in I, s\in S\right)
\]
holds. Since $\mathcal{F}$ was assumed to be balanced as well, we have
\[
\xi_{1-s}\left(\xi_{s}(x)\right)=x
\]
for all possible $x\in I$ and $s\in S$. Using the above representation, this means that the functions $u$ and $v$ have to satisfy
\[
u(u(x)+v(s))= x-v(1-s)
\qquad
\left(x\in I, s\in S\right).
\]
The functions $u$ and $v$ are continuous and strictly increasing, thus the latter identity implies that
\[
u(x+s)= u^{-1}(x)-v(1-v^{-1}(s))
\qquad
\left({x\in u(I), s\in v(S)}\right),
\]
which is a Pexider equation. Thus, we have especially that
\[
u(x)= \alpha_{1}x+\beta_{1}
\quad 
\left(x\in I\right)
\qquad
\text{and}
\quad
u^{-1}(x)= \alpha_{1}x+\beta_{2}
\qquad
\left({x\in u(I)}\right)
\]
with some appropriate constants $\alpha_{1}, \beta_{1}, \beta_{2}$. This is possible however only if $\alpha_{1}=1$ and $\beta_{2}=-\beta_{1}$.
So
\[
v(s)+2\beta_{1}=-v(1-s)
\qquad
\left(s\in S\right).
\]
This means that the function $\nu \colon S\to \mathbb{R}$ defined by
\[
\nu(s)= v(s)+\beta_{1}
\qquad
\left(s\in S\right)
\]
is antisymmetric with respect to the point $s_{0}=\frac{1}{2}$.
Therefore,
\[
\xi_{s}(x)= u(x)+v(s)= \left[x+\beta_{1}\right]+\left[\nu(s)-\beta_{1}\right]= x+\nu(s)
\]
for all $x\in I$ and $s\in S$, as stated.
\end{proof}

\subsection{Multiplicatively translational mappings}

In this subsection, we investigate Iverson's law of similarity
\[
\xi_{s}(\lambda x)= \gamma(\lambda , s)\xi_{\eta(\lambda, s)}(x)
\qquad
\left(x\in I, \lambda\in J, s\in S\right)
\]
assuming that the two-variable mapping $\eta$ is multiplicatively translational.

\begin{dfn}
Let $S, J\subset \mathbb{R}$ be  intervals such that $0\in S$ and $1\in J$. A mapping
$\eta\colon J\times S\to S$ is called \emph{multiplicatively translational} if
\[
\eta(\lambda\tilde{\lambda}, s)= \eta(\tilde{\lambda}, \eta(\lambda, s))
\]
and also the boundary conditions
\[
\eta(\lambda, 0)=0
\qquad
\eta(1, s)=s
\]
hold for all $\lambda, \tilde{\lambda}\in J$ and $s\in S$.
\end{dfn}

A source of motivation for studying multiplicatively translational $\eta$ comes from the shift invariance relationship 
\[
\xi_{\lambda^\theta s}(\lambda x) = \lambda\, \xi_s (x)  
\]
appearing in the work of \cite{PavelIverson1981}. In this context, $\xi_s (x)$ represents the intensity of an unmasked tone that matches the loudness of a tone of intensity $x$ embedded in a background of intensity $s$.  We could write this shift invariance relationship as 
\[
 \xi_s (\lambda x) =  \lambda\, \xi_{\lambda^{-\theta} s}(x)\,,
\]
which is clearly a case of Iverson's law of similarity with $\eta (\lambda, s)=\lambda^{-\theta}s$.  Note that the multiplicatively translational property generalizes the transformation $\eta (\lambda, s)=\lambda^{-\theta} s$,   since we have for this transformation that  
\[
\eta (\tilde{\lambda}, \eta(\lambda, s)) = \eta (\tilde{\lambda}, \lambda^{-\theta} s) = \tilde{\lambda}^{-\theta} \lambda^{-\theta} s =\eta(\lambda\tilde{\lambda}, s)\,.
\]
We also point out that in this experimental context, it is reasonable that the shift invariance relationship would hold for $s=0$. 

Note the following relationship between Iverson's law of similarity and the multiplicatively translational property, also pointed out in \cite{HsuIverson2016}.  If we
write $\tilde{\lambda} x$ in place of $x$ in Iverson's law of similarity, then by a successive application of the similarity law,
\begin{multline*}
\xi_{s}(\lambda \tilde{\lambda} x)= \gamma(\lambda , s)\xi_{\eta(\lambda, s)}( \tilde{\lambda}x)
\\
=
\gamma(\lambda , s)\gamma(\tilde{\lambda}, \eta(\lambda, s)) \xi_{\eta(\tilde{\lambda}, \eta(\lambda, s))}(x)
\quad
\left(\lambda, \tilde{\lambda}\in J, s\in S\right)
\end{multline*}
follows.
On the other hand, if we use Iverson's law of similarity only once, but for $\lambda \tilde{\lambda}$, we obtain that
\[
\xi_{s}(\lambda \tilde{\lambda} x)= \gamma(\lambda \tilde{\lambda} , s)\xi_{\eta(\lambda \tilde{\lambda} , s)}(x)
\quad
\left(\lambda, \tilde{\lambda}\in J, s\in S\right).
\]
Thus necessarily,
\[
\gamma(\lambda , s)\gamma(\tilde{\lambda}, \eta(\lambda, s)) \xi_{\eta(\tilde{\lambda}, \eta(\lambda, s))}(x)
=
\gamma(\lambda \tilde{\lambda} , s)\xi_{\eta(\lambda \tilde{\lambda} , s)}(x)
\quad
\left(\lambda, \tilde{\lambda}\in J, s\in S\right). 
\]
Therefore, the assumption that $\eta$ is multiplicatively translational expresses in some sense a kind of consistency in the above computations.


\begin{lem}\label{lem_trans}
Let $S, J\subset \mathbb{R}$ be  intervals such that $0\in S$ and $1\in J$. If the mapping
$\eta\colon J\times S\to S$ is multiplicatively translational and there exists an $s^{\ast}\in S$ such that the mapping
\[
J\ni \lambda \longmapsto \eta(\lambda, s^{\ast})
\]
is bijective, then there exists a bijective function $H\colon J\to S$ with $H(0)=0$
such that
\[
\eta(\lambda, s)= H(\lambda\cdot H^{-1}(s))
\qquad
\left(\lambda\in J, s\in S\right).
\]
\end{lem}
\begin{proof}
Define the function $H\colon J\to S$ through
\[
H(\lambda)= \eta(\lambda, s^{\ast})
\qquad
\left(\lambda\in J\right).
\]
Then $H$ is a bijection from $J$ onto $S$. Further, since $\eta$ is multiplicatively translational,
\[
H(\lambda \tilde{\lambda})= \eta(\lambda, H(\tilde{\lambda}))
\]
holds for all $\lambda, \tilde{\lambda}\in J$, from which
\[
\eta(\lambda, s)= H(\lambda\cdot H^{-1}(s))
\qquad
\left(\lambda\in J, s\in S\right)
\]
follows. Moreover, we also have
\[
0= \eta(\lambda, 0)= H(\lambda\cdot H^{-1}(0))
\qquad
\left(\lambda \in J\right),
\]
from which we deduce
\[
H^{-1}(0)= \lambda H^{-1}(0)
\]
for all $\lambda \in J$, so $H^{-1}(0)=0$ and also $H(0)=0$.
\end{proof}

\begin{rem}
Observe that from the previous proof, it follows that either $J= \left\{ 1\right\}$ and thus $S= \left\{ 0\right\}$ follows, or $0\in J$ holds. So $[0, 1]\subset J$, otherwise the function $\eta$ is defined on the singleton $\left\{ (1, 0)\right\}$. 
\end{rem}

\begin{rem}
We further note that in the above proof, the boundary conditions only play a role in that with their help we can show that $H(0)=0$. This means that if we omit these boundary conditions, the representation
\[
 \eta(\lambda, s)= H(\lambda \cdot H^{-1}(s))
\]
is still true. 
\end{rem}

\begin{rem}
 Multiplicatively translational mappings were previously studied in  \cite{HsuIverson2016} and even in  \cite{Aczel1966, Aczel1987}. However, we would like to emphasize that the conditions there are different from the ones we use. 
 Indeed, according to \cite[Section 6, Theorem 3]{Aczel1987}, if $S$ is an interval and $\eta\colon ]0, +\infty[\times S\to S$ is a function that is \emph{continuous in both variables} and for which the mapping 
 \[
  \lambda \to \eta(\lambda, s)
 \]
is \emph{nonconstant for every $s\in S$}, and further
\[
 \eta(\tilde{\lambda}\lambda, s)= \eta(\lambda, \eta(\lambda, s))
\]
holds for all $\tilde{\lambda}, \lambda$ and $s\in S$, then there exists a continuous and strictly monotonic function $\psi\colon ]0, +\infty[\to S$ such that 
\[
 \eta(\lambda, s)= \psi(\lambda \cdot \psi^{-1}(s)) 
 \qquad 
 \left(\lambda \in ]0, +\infty[, s\in S\right). 
\]
Although the representations are the same, in this case it is assumed that $J=]0, +\infty[$ and the regularity conditions are different. 
\end{rem}

In what follows we will consider only those multiplicatively translational mappings $\eta$ for which there exists an $s^{\ast}\in S$ such that the mapping
\[
J\ni \lambda \longmapsto \eta(\lambda, s^{\ast})
\]
is bijective. In other words, we assume that there exists a bijective function $H\colon J\to S$ with $H(0)=0$ such that
\[
\eta(\lambda, s)= H(\lambda\cdot H^{-1}(s))
\qquad
\left(\lambda \in J, s\in S\right).
\]
For the sake of simplicity and easier distinction, such mappings will be termed \emph{regular multiplicatively translational} mappings. 

Observe that in this case the assumptions of Remark \ref{form_xi} are fulfilled. 
Indeed, a careful adaption of that proof yields that we have 
\[
\xi_{s^{\ast}}\left(\frac{H^{-1}(s)}{H^{-1}(s^{\ast})} x\right)= \gamma\left( \frac{H^{-1}(s)}{H^{-1}(s^{\ast})}, s^{\ast}\right) \xi_{s}(x)
\qquad
\left(x\in I, s \in S\right).
\]

\begin{prop}\label{prop2}
Let $I, J, S\subset \mathbb{R}$ be  intervals with $I\cdot J= \left\{ \lambda x\, \vert \, x\in I, \lambda \in J \right\}\subset I$ and $1\in J$, $0, 1\in S$. Let further $\gamma, \eta \colon J\times S\to S$ be functions and assume that the one-parameter family of functions $\xi_{s}\colon I\to \mathbb{R}{\setminus \left\{ 0\right\}}\, (s\in S)$ fulfills
\[
\xi_{s}(\lambda x)= \gamma(\lambda , s)\xi_{\eta(\lambda, s)}(x)
\qquad
\left(x\in I, \lambda\in J, s\in S\right).
\]
If $\eta$ is a regular multiplicatively translational mapping then
\[
\gamma(\lambda, s)= \frac{\kappa (h(s)\lambda)}{\kappa(h(s))}
\qquad
\left(\lambda, s\in S\right)
\]
holds, except for the points where the denominator vanishes, where the functions $\kappa$ and $h$ are defined by
\[
\kappa(\lambda)= \gamma(\lambda, 1)
\qquad
\text{and}
\qquad
h(s)= \frac{H^{-1}(s)}{H^{-1}(1)}
\qquad
\left(\lambda, s\in S\right).
\]
\end{prop}

\begin{proof}
Under the assumptions of the proposition, let $\tilde{\lambda} \in J$ be arbitrary and let us substitute $\tilde{\lambda} x$ in place of $x$ in Iverson's law of similarity to get
\begin{multline*}
\xi_{s}(\lambda \tilde{\lambda} x)= \gamma(\lambda , s)\xi_{\eta(\lambda, s)}(\tilde{\lambda} x)
\\=
\gamma(\lambda , s) \gamma(\tilde{\lambda} , \eta(\lambda, s)) \xi_{\eta(\tilde{\lambda, \eta(\lambda, s)})}(x)
\qquad
\left(x\in I, \lambda, \tilde{\lambda}\in J, s\in S\right).
\end{multline*}
On the other hand, Iverson's law of similarity with $\lambda \tilde{\lambda}$ instead of $\lambda$ is
\[
\xi_{s}(\lambda \tilde{\lambda} x)= \gamma(\lambda \tilde{\lambda} , s)\xi_{\eta(\lambda \tilde{\lambda}, s)}(x)
\qquad
\left(x\in I, \lambda, \tilde{\lambda}\in J, s\in S\right).
\]
Using that the mapping $\eta$ is multiplicatively translational, we get that
\[
\gamma(\lambda \tilde{\lambda} , s)= \gamma(\lambda , s)\cdot \gamma(\tilde{\lambda} , \eta(\lambda, s))
\]
holds for all $ \lambda, \tilde{\lambda}\in J, s\in S$. Further, since $\eta$ is not only multiplicatively translational but also regular multiplicatively translational, we have
\[
\gamma(\lambda \tilde{\lambda} , s)= \gamma(\lambda , s) \cdot \gamma(\tilde{\lambda} , H(\lambda \cdot H^{-1}(s)))
\qquad
\left(\lambda, \tilde{\lambda}\in J, s\in S\right).
\]
Let us substitute $s=1$  and define the functions $h, \kappa$ on $S$ by
\[
h(s)= \frac{H^{-1}(s)}{H^{-1}(1)}
\qquad
\text{and}
\qquad
\kappa(\lambda)= \gamma (\lambda, 1)
\qquad
\left(\lambda\in J, s\in S\right)
\]
to deduce that
\[
\kappa(h(s)\lambda)= \kappa(h(s))\cdot \gamma(\lambda, s)
\qquad
\left(\lambda\in J, s\in S\right),
\]
that is,
\[
\gamma(\lambda, s)= \frac{\kappa(h(s)\lambda)}{\kappa(h(s))}
\]
holds for all $\lambda, s\in S$.
\end{proof}

\begin{rem}\label{cor1}
Let $I, J, S\subset \mathbb{R}$ be  intervals with $1\in I$, $0, 1\in J$, $0, 1\in S$, and $I\cdot J= \left\{ \lambda x\, \vert \, x\in I, \lambda \in J \right\}\allowbreak\subset I$. Let further $\eta \colon J\times S\to S$ and $\phi \colon S\to \mathbb{R}$ be functions and assume that the one-parameter family of functions $\xi_{s}\colon I\to \mathbb{R}\, (s\in S)$ fulfills the following power law
\[
\xi_{s}(\lambda x)= \lambda^{\phi(s)}\xi_{\eta(\lambda, s)}(x)
\qquad
\left(x\in I, \lambda\in J, s\in S\right).
\]
If $\eta$ is a regular multiplicatively translational mapping, then there exists a function $F\colon I\to \mathbb{R}$ such that
\[
\xi_{s}(x)= x^{\phi(s)}F(x\cdot H^{-1}(s))
\qquad
\left(x\in I, s\in S\right).
\]
In this case we automatically have that 
\[
\eta(\lambda, s)= H(\lambda \cdot H^{-1}(s))
\qquad
\text{and}
\qquad
\gamma(\lambda, s)=\lambda^{\phi(s)}
\qquad
\left(\lambda\in J, s\in S\right)
\]
with some function $H\colon J\to S$ such that $H(0)=0$.
Further,
\[
\xi_{s}(\lambda x)= \lambda^{\phi(s)} \xi_{H(\lambda \cdot H^{-1}(s))} (x)
\qquad
\left(x\in I, \lambda\in J, s\in S\right).
\]
With $x=1$, we get that 
\[
\xi_{s}(\lambda )= \lambda^{\phi(s)} \xi_{H(\lambda \cdot H^{-1}(s))} (1)
\qquad
\left(x\in I, \lambda\in J, s\in S\right).
\]
Define the function $F\colon I\to \mathbb{R}$ by
\[
F(x)= \xi_{H(x)}(1)
\qquad
\left(x\in I\right)
\]
to get that
\[
\xi_{s}(x)= x^{\phi(s)} F(x \cdot H^{-1}(s))
\qquad
\left(x\in I, s\in S\right).
\]
\end{rem}

In the statement below we consider the power law
\[
\xi_{s}(\lambda x)= \lambda^{\phi(s)}\xi_{\eta(\lambda, s)}(x)
\qquad
\left(x\in I, \lambda\in J, s\in S\right)
\]
with a \emph{monotonic} function $\phi$.

\begin{cor}\label{cor2}
Let $\phi\colon S\to S$ be a \emph{monotonic} function on the interval $S$. 
Suppose  the one-parameter family of functions $\xi_{s}\colon I\to \mathbb{R}{\setminus \left\{ 0\right\}}\, (s\in S)$ fulfills the following power law
\[
\xi_{s}(\lambda x)= \lambda^{\phi(s)}\xi_{\eta(\lambda, s)}(x)
\qquad
\left(x\in I, \lambda\in J, s\in S\right)
\]
with a regular multiplicatively translational mapping $\eta \colon J\times S \to S$. 
Then there exists a countable partition $(S_{\alpha})_{\alpha\in A}$ of subintervals of $S$ such that
for all $\alpha \in A$, the interval $S_{\alpha}$ cannot have an open subinterval $\tilde{S}$ with positive length such that $\phi$ is \emph{strictly} monotonic on $\tilde{S}$. 
\end{cor}

\begin{proof}
If the function  $\phi$ is monotonic on the interval $S$, then the points of discontinuity of the function $\phi$ form a countable subset in $S$. Further this countable set induces a countable partition $(S_{\alpha})_{\alpha\in A}$ of the interval $S$ such that for all $\alpha\in A$, the function $\phi$ is continuous on $S_{\alpha}^{\circ}$ (on the interior of the interval $S_{\alpha}$).

Assume to the contrary that there exists an $\alpha\in A$ such that the interval $S_{\alpha}$ has a proper open subinterval $\tilde{S}$ with positive length such that $\phi$ is strictly monotonic (and due to the construction, continuous) on $\tilde{S}$.

Due to Iverson's law of similarity, we have
\[\xi_s(\lambda \tilde{\lambda} x) = (\lambda \tilde{\lambda})^{\phi(s)} \xi_{\eta(\lambda \tilde{\lambda}, s)}(x)
\qquad
\left(x\in I, \lambda, \tilde{\lambda}\in J, s\in \tilde{S}\right). \]
On the other hand, we also have
\begin{multline*}
\xi_s(\lambda \tilde{\lambda} x)=
\xi_{s}(\lambda (\tilde{\lambda x}))
= \lambda^{\phi(s)} \cdot \xi_{\eta(\lambda, s)}(\tilde{\lambda} x)
\\
= \lambda^{\phi(s)}\cdot \tilde{\lambda}^{\phi(\eta(\lambda, s))} \xi_{\eta(\tilde{\lambda}, \eta(\lambda, s))}(x)
\quad
\left(x\in I, \lambda, \tilde{\lambda}\in J, s\in \tilde{S}\right).
\end{multline*}
Since $\eta$ satisfies the regular multiplicative translation property, we have
\[
(\lambda \tilde{\lambda})^{\phi(s)} = \lambda^{\phi(s)} \cdot \tilde{\lambda}^{\phi(\eta(\lambda, s))},
\]
that is,
\[
\phi(s)= \phi(\eta(\lambda, s))
\]
for all $\lambda\in J, s\in \tilde{S}$.
At the same time, the function $\phi$ was assumed to be continuous and strictly monotonic on $\tilde{S}$. Thus
\[
s= \eta(\lambda, s)
\]
holds for all $\lambda\in J, s\in \tilde{S}$.
Since $\eta$ is a regular multiplicatively translational mapping,
\[
\eta(\lambda, s)= H(\lambda \cdot H^{-1}(s))
\]
holds for all $\lambda\in J, s\in S$ with an appropriate function $H$.
This however means that we have
\[
s= H(\lambda \cdot H^{-1}(s)),
\]
i.e.,
\[
H^{-1}(s)= \lambda \cdot H^{-1}(s)
\]
for all $\lambda\in J, s\in \tilde{S}$ implying that $\lambda =1$ holds for all $\lambda \in J$. This is a contradiction. Thus $\phi$ cannot be strictly monotonic on $\tilde{S}$. 
\end{proof}

\begin{rem}
 Notice that the assumption that the function  $\eta$ is a  \emph{regularly} multiplicative mapping was used only at the end of the proof of Corollary \ref{cor1}. Indeed, to deduce that we have 
 \[
  \phi(s)= \phi(\eta(\lambda, s)), 
 \]
we used only that $\eta$ fulfills $\eta(\lambda \tilde{\lambda}, s)= \eta(\lambda, \eta(\tilde{\lambda}, s))$. 
 In case $\eta(\lambda, s)= s$ for all possible values $\lambda$ and $s$, we have Falmagne's power law for the sensitivity functions $\xi_{s}\, (s\in S)$, i.e.,  
 \[
  \xi_{s}(\lambda x)= \lambda^{\phi(s)} \xi_{s}(x) 
  \qquad 
  \left(x\in I, s\in S\right)
 \]
 holds.
 Recall that $\xi$ is termed to be \emph{weakly balanced} if for all $x\in I$ we have $\xi_{\frac{1}{2}}(x)=x$.  
 According to \cite[Theorem 6]{HsuIverson2016}, if, in addition to the above, $\xi$ admits a weakly balanced  affine representation and $\phi$ is continuous, strictly monotonic and nowhere zero, then $\xi$ has a Fechnerian representation. 
 The mentioned \cite{HsuIverson2016} paper also contains results about the power law 
 \[
  \xi_{s}(\lambda x)= \lambda^{\phi(s)} \xi_{\eta(\lambda, s)}(x) 
  \qquad 
  \left(x\in I, s\in S\right)
 \]
with a multiplicatively translational mapping $\eta$. Recall that in the case where $\eta$ is regularly multiplicatively translational, the conditions of Remark \ref{form_xi} hold. Thus, with the help of Proposition \ref{prop1}, we can specially obtain those sensitivities that satisfy the above power law and have an affine representation. In such a way we deduce Corollary 7 of \cite{HsuIverson2016} without assuming continuous differentiability about the scales $u$ and $\sigma$ and twice differentiability about $\xi$. 
 \end{rem}


According to Remark \ref{cor1}, if the one-parameter family of functions $\xi$ satisfies the power law 
\[
 \xi_{s}(\lambda x)= \lambda^{\phi(s)}\xi_{\eta(\lambda, s)}(x)
\]
with some regular multiplicatively translational mapping $\eta$, then $\xi$ can be written in the  form
\[
 \xi_{s}(x)= x^{\phi(s)}F(x H^{-1}(s)). 
\]
Afterwards, we showed that if $\phi$ is \emph{monotone}, then the interval $S$ cannot have a proper open subinterval on which the function $\phi$ is strictly monotone. This motivates the study of one-parameter families of the form  
\[
 \xi_{s}(x)= x^{r}F(x H^{-1}(s)).
\]

In the following, we examine whether one-parameter families of this form  
can have a subtractive representation. As we shall see, the answer is yes. Note that one-parameter families of this form automatically satisfy the above power law, with the choice $\eta(\lambda, s)=H(\lambda \cdot H^{-1}(s))$ and $\phi(s)= r$. In this case $\eta$ is obviously multiplicatively translational, but not necessarily \emph{regular} multiplicatively translational.

\begin{prop}\label{prop_subtractive}
 Let $I$ and $S$ be real intervals of positive length and suppose that the one-parameter family of functions $\xi_{s}\colon I\to \mathbb{R}\, (s\in S)$  can be represented as 
 \begin{equation}\label{rep1}
  \xi_{s}(x)= x^{r}F(x\cdot H^{-1}(s)) 
  \qquad 
  \left(x\in I, s\in S\right)
 \end{equation}
and also as 
\begin{equation}\label{subtr}
 \xi_{s}(x)= u^{-1}(s+w(x)) 
 \qquad 
 \left(x\in I, s\in S\right)
\end{equation}
with an appropriate {nonzero} real constant $r$, continuous and strictly increasing functions $u, w$ and $H$ and with a continuous function $F$. If the domain of the functions $u, w$ and $H$ and the range of the function $F$ is contained in the set of positive reals, then the following cases are possible. 
\begin{enumerate}[{Case} I]
 \item \[
 \xi_{s}(x)= a  e^{b\rho s} x^{r+\rho}
 \qquad 
 \left(x\in I, s\in S\right)
\]
with appropriate constants $a$, $b$ and $\rho$. 
\item 
\[
 \xi_{s}(x)= a\left(cx^{\frac{r}{\rho}}+s-\varepsilon\right)^{\rho}
 \qquad 
 \left(x\in I, s\in S\right)
\]
with appropriate constants $a, \varepsilon, \rho$ and $c$. 
\end{enumerate}

\end{prop}

\begin{proof}
The assumption that the domain of the functions $u, w$ and $H$ and the range of the function $F$ is contained in the set of positive reals guarantees that we can take the logarithm of both the sides of Equation \eqref{rep1}. In this case 
 \begin{align*} 
   \log(\xi_{s}(x))&= r\log(x)+ \log(F(x\cdot H^{-1}(s)))\\
   & =r\log(x)+ \log\circ F\circ \exp(\log(x)+ \log\circ H^{-1}(s))
 \end{align*} 
follows for all $x\in I$ and $s\in S$. 
Further, if we take the logarithm of both the sides of \eqref{subtr}, we deduce 
\[
 \log(\xi_{s}(x))= \log\circ u^{-1}(s+w(x))  
 \qquad 
 \left(x\in I, s\in S\right). 
\]
Therefore we have 
\[
 \log\circ u^{-1}(s+w(x))= r\cdot \log(x)+ \log\circ F\circ \exp(\log(x)+ \log\circ H^{-1}(s)), 
\]
that is, 
\[
  \log\circ u^{-1}(s+y)= r\cdot \log\circ w^{-1}(y)+ \log\circ F\circ \exp(\log\circ w^{-1}(y)+ \log\circ H^{-1}(s))
\]
for all $y\in J$ and $s\in S$, where the interval $J$ denotes the range of the function $w$. 
Let us introduce the functions $h, m, f, \ell$ and $g$ by 
\begin{align*}
 h&=\log \circ u^{-1}\\
 m&= -r\cdot \log w^{-1}\\
 f&= \log \circ F\circ \exp\\
 \ell& = \log \circ w^{-1}\\
 g&= \log \circ H^{-1}
\end{align*}
to obtain that the latter equation takes the form 
\[
 f(\ell(y)+g(s))= m(y)+ h(y+s) 
 \qquad 
 \left(y\in J, s\in S\right). 
\]
Since $h, \ell$ and $m$ are philandering, due to the respective properties of the functions $u$ and $w$, Lemma \ref{lem_Lundberg} can be applied. 
However, let us note that in our case with the above notations, we have $ r\cdot \ell +m \equiv 0$. This shows that Cases IV and V are impossible in our case. 
\begin{enumerate}[{Case} I]
 \item In this case 
 \[
  m(x)= \rho \ell(x)-\rho b x+ \alpha +\rho \beta + \tau,
 \]
at the same time, due to $ r\cdot l +m \equiv 0$, {so 
\[
 -r\ell(x)= \rho \ell(x)-\rho b x+ \alpha +\rho \beta + \tau
\]
holds. From this
\[
 \ell(x)= \frac{\rho b}{r+\rho}x - \frac{\alpha+\rho \beta +\tau}{r+\rho}, 
\]
and 
\[
 w(x)= \frac{r+\rho}{\rho b}\log (x)+ \frac{\alpha+\rho \beta +\tau}{\rho b}
\]
follows. }
Further, 
\[
 h(x)= -\tau +\rho b x, 
\]
so 
\[
 u^{-1}(x)= e^{-\tau}e^{\rho b x}. 
\]
Therefore, 
\[
 \xi_{s}(x)= \exp(\beta \rho+\log(x)\rho+bs\rho+\alpha+r\log(x))
 =
 e^{\beta \rho+\alpha }e^{b\rho s}x^{r+\rho}
\]
holds. 

\item In this case 
\[
 h(x)= -\tau +\alpha+ \rho \log(bc+de^{\delta x}), 
\]
yielding that 
\[
 u^{-1}(x)= e^{-\tau +\alpha}\left(bc+de^{\delta x}\right)^{\rho}
\]
and 
\[
 \ell(x)= -\frac{1}{\kappa}\log \left(\beta +b e^{-\delta x}\right). 
\]
From this, 
\[
 w^{-1}(x)=\left(\beta +b e^{-\delta x}\right) ^{-\frac{1}{\kappa}}, 
\]
that is, 
\[
 w(x)= -\frac{1}{\delta}\log \left(\frac{x^{-\kappa}-\beta}{b}\right)
\]
follows. 
From these, we obtain 
\[
 \xi_{s}(x)= e^{-\tau+\alpha}\left(bc+ d e^{\delta s} \frac{b}{x^{-\kappa}-\beta}\right)^{\rho}
 = 
 e^{\alpha}\left(e^{-\frac{\tau}{\rho}}bc+ de^{-\frac{\tau}{\rho}} e^{\delta s} \frac{b}{x^{-\kappa}-\beta}\right)^{\rho}. 
\]
At the same time, since 
\[
 g(x)= \frac{1}{\kappa}\log(-\beta c+de^{\delta x}), 
\]
we get that 
\[
 H^{-1}(s)= \left(-\beta c+de^{\delta s}\right)^{\frac{1}{\kappa}}. 
\]
Further, 
\[
 f(x)= \alpha+ \rho \log (c+ e^{\kappa x}), 
\]
we obtain that 
\[
 F(x)= e^{\alpha}(c+x^{\kappa})^{\rho}. 
\]
Therefore 
\begin{multline*}
 \xi_{s}(x)= 
 x^re^\alpha((x(d e^{\delta s}-c\beta)^{1/\kappa})^\kappa+c)^\rho
 =
 x^re^\alpha(x^{\kappa}(d e^{\delta s}-c\beta)+c)^\rho
 \\
 =
 e^\alpha(x^{\kappa+\frac{r}{\rho}}(d e^{\delta s}-c\beta)+cx^{\frac{1}{\rho}})^\rho
\end{multline*}

Again, the obtained representations for $\xi$ have to agree everywhere, so 
\[
 e^{\alpha}\left(e^{-\frac{\tau}{\rho}}bc+ de^{-\frac{\tau}{\rho}} e^{\delta s} \frac{b}{x^{-\kappa}-\beta}\right)^{\rho}
 =
 e^\alpha(x^{\kappa+\frac{r}{\rho}}(d e^{\delta s}-c\beta)+cx^{\frac{1}{\rho}})^\rho, 
\]
from which we obtain that 
\[
 e^{-\frac{\tau}{\rho}}bc+ de^{-\frac{\tau}{\rho}} e^{\delta s} \frac{b}{x^{-\kappa}-\beta}
 =
 x^{\kappa+\frac{r}{\rho}}(d e^{\delta s}-c\beta)+cx^{\frac{1}{\rho}}
\]
should hold for all possible $x$ and $s$. Differentiating both sides with respect to the variable $s$, 
\[
 e^{-\frac{\tau}{\rho}} \frac{b}{x^{-\kappa}-\beta}=
 x^{\kappa+\frac{r}{\rho}}
\]
follows, after simplification with $d\delta$. Transforming this further, we get that 
\[
 e^{-\frac{\tau}{\rho}} b =
 x^{\frac{r}{\rho}}-\beta x^{\kappa+\frac{r}{\rho}}. 
\]
{
In other words, 
\[
1\cdot x^{\frac{r}{\rho}}-\beta \cdot  x^{\kappa+\frac{r}{\rho}}-b\cdot e^{-\frac{\tau}{\rho}}=0. 
\]
As monomials corresponding to different powers are linearly independent, this is possible only if some powers are equal, or all the coefficients are zero. Since none of the coefficients can be zero, some powers should be equal. At the same time, $\kappa\neq 0$ and $r\neq 0$ by our assumptions. So this case is impossible.  }
\item In case of the last possibility 
\[
h(x)= \rho \log (bd\alpha x+ \alpha \varepsilon -b\beta) 
\]
so 
\[
 u^{-1}(x)= \left(bd\alpha x+ \alpha \varepsilon -b\right)^{\rho}
\]
and since 
\[
 \ell(x)= -\frac{1}{\kappa}\log (\varepsilon +bdx), 
\]
we have 
\[
 w(x)= \frac{x^{-\kappa}-\varepsilon}{bd}. 
\]
Therefore, 
\[
 \xi_{s}(x)= \left(bd\alpha s+ \alpha x^{-\kappa}-b\beta\right)^{\rho}. 
\]
At the same time, since 
\[
 g(x)= \frac{1}{\kappa}\log \left(\beta -bd x\right), 
\]
we get that 
\[
 H^{-1}(s)= \left(\beta -bd s\right)^{\frac{1}{\kappa}}
 \]
and 
\[
 f(x)= \rho \log \left(\alpha-be^{\kappa x}\right), 
\]
thus 
\[
 F(x)= \left(\alpha -bx^{\kappa}\right)^{\rho}. 
\]
So 
\begin{align*}
 \xi_{s}(x)&= x^{r}\left(\alpha- bx^{\kappa}(\beta-bds)\right)^{\rho}\\
 &=\left(\alpha x^{\frac{r}{\rho}}-b x^{\frac{r}{\rho}+\kappa}(\beta-bds)\right)^{\rho}\\
 &= \left(\alpha x^{\frac{r}{\rho}}-b\beta+ b^{2}ds\right)^{\rho}, 
\end{align*}
where we also used that $\frac{r}{\rho}+\kappa=0$. The latter relationship results from the equation $r\cdot \ell+m\equiv 0$.
Therefore, 
\[
 \xi_{s}(x)= a\left(cx^{\frac{r}{\rho}}+s-\varepsilon\right)^{\rho}
\]
holds with some appropriate constants. 
\end{enumerate}
\end{proof}

Regarding the assumptions of Proposition \ref{prop_subtractive}, note that the domain of the functions $u, w$ and $H$ and the range of the function $F$ will be contained in the set of positive reals if the sets $I$ and $J$ are contained in the set of positive reals and the outputs of $\xi_s$ are positive.  These are reasonable assumptions in models of intensity discrimination.

{
Applying the above proof to the function $w=u$, we obtain the sensitivity functions that have a Fechnerian representation.
}
{
\begin{cor}
 Let $I$ and $S$ be real intervals of positive length and suppose that the one-parameter family of functions $\xi_{s}\colon I\to \mathbb{R}\, (s\in S)$  can be represented as 
 \begin{equation*}
  \xi_{s}(x)= x^{r}F(x\cdot H^{-1}(s)) 
  \qquad 
  \left(x\in I, s\in S\right)
 \end{equation*}
and also as 
\begin{equation}\label{Fec}
 \xi_{s}(x)= u^{-1}(s+u(x)) 
 \qquad 
 \left(x\in I, s\in S\right)
\end{equation}
with an appropriate real constant $r$, continuous and strictly increasing functions $u$ and $H$ and with a continuous function $F$. If the domain of the functions $u$ and $H$ and the range of the function $F$ is contained in the set of positive reals, then 
 \[
 \xi_{s}(x)= e^{\bar{\rho} s} \cdot  x
 \qquad 
 \left(x\in I, s\in S\right)
\]
holds with an appropriate constant $\bar{\rho}$. 
\end{cor}
}
\begin{proof}
If the one-parameter family of functions $\xi_{s}\colon I\to \mathbb{R}\, (s\in S)$ admits a Fechnerian representation \eqref{Fec}, then it has a subtractive representation \eqref{subtr} with the function $w=u$. Thus, in view of the proof of Proposition \ref{prop_subtractive} we have following cases. 
\begin{enumerate}[{Case} I]
    \item \[ w(x)= \frac{r+\rho}{\rho b}\log (x)+ \frac{\alpha+\rho \beta +\tau}{\rho b} \qquad  \text{and}\qquad  u(x)= \frac{\log(x)}{\rho b}+\frac{\tau}{\rho b}\]
    At the same time $w=u$, which in this case is possible only if $\alpha +\rho b=0$ and $r+\rho=1$ hold. This means however that 
    \[
    \xi_{s}(x)= e^{b\rho s}\cdot x 
    \qquad 
    \left(x\in I, s\in S\right)
    \]
    Letting $\bar{\rho}= b\rho$ we obtain the statement.  
    \item The other possibility is that we have Case II in Proposition \ref{prop_subtractive}, that is, 
    \[
 u^{-1}(x)= \left(bd\alpha x+ \alpha \varepsilon -b\right)^{\rho}
\]
and  
\[
 w(x)= \frac{x^{-\kappa}-\varepsilon}{bd}. 
\]
In this case however 
\[
u(x)= \frac{x^{\frac{1}{\rho}}+b-\alpha \varepsilon}{bd \alpha}. 
\]
As $b\neq 0$ should hold, $w=u$ in this case is impossible. 
\end{enumerate}
\end{proof}

\begin{rem}
 Observe that the mapping $\xi$ obtained in Case I of Proposition \ref{prop_subtractive} corresponds to Case I of Theorem 10 in \cite{DobleHsu2020} if we choose there the function $\eta$ to be 
 \[
  \eta(\lambda, s)= s+ \frac{1}{b}\log(\lambda). 
 \]
 Case II agrees with Case V of Theorem 10 in \cite{DobleHsu2020}.
 The constants are denoted differently here than in the mentioned paper, but the mentioned representations are the same. 

 Regarding Proposition \ref{prop_subtractive}, it is very important to emphasize that while regular multiplicatively translational mappings are considered in the former, we only assume that the mapping in question is multiplicatively translational mapping in the latter. 
 As we saw in the previous remark, Case II of Proposition \ref{prop_subtractive}  corresponds to Case V of Theorem 10 in \cite{DobleHsu2020}. In that case, we have 
 \[
  \eta(\lambda, s)= \lambda^{-\delta}(s-\varepsilon)+\varepsilon. 
 \]
Observe that we have 
\begin{multline*}
 \eta(\lambda, \eta(\tilde{\lambda}, s) )
 = 
 \lambda^{-\delta}(\eta(\tilde{\lambda}, s)-\varepsilon)+\varepsilon
 \\
 =
 \lambda^{-\delta}((\tilde{\lambda}^{-\delta}(s-\varepsilon)+\varepsilon)-\varepsilon)+\varepsilon
 =
 \lambda^{-\delta}\tilde{\lambda}^{-\delta}(s-\varepsilon)+\varepsilon
 \\=
 \left(\lambda\tilde{\lambda}\right)^{-\delta}\cdot (s-\varepsilon)+\varepsilon
 =
 \eta(\lambda \tilde{\lambda}, s)
\end{multline*}
for all possible $\lambda, \tilde{\lambda}$ and $s$. So $\eta$ is multiplicatively translational. Further, we have 
\[
 \eta(1, s)= 1^{-\delta}(s-\varepsilon)+\varepsilon= s 
 \qquad 
 \left(s\in S\right). 
\]
At the same time, the boundary condition $\eta(\lambda, 0)=0$ holds if and only if 
\[
 0= \eta(\lambda, 0)= \lambda^{-\delta}(0-\varepsilon)+\varepsilon= \varepsilon (1-\lambda^{-\delta}), 
\]
that is, if and only if $\varepsilon=0$ or $\delta=0$. Assuming $\varepsilon=0$, we have 
\[
 \eta(\lambda, s)= \lambda^{-\delta}s. 
\]
The case when $\delta$ is negative (i.e. when $-\delta$ is positive) is studied in Corollary \ref{Cor_shift} with $\theta= -\delta$ (see below). 

If $\delta$ is positive (i.e. when $-\delta$ is negative), then $\eta$ is although multiplicatively translational, but not \emph{regular} multiplicatively translational. Indeed, in this case $J= ]0, +\infty[$ and $S= [0, +\infty[$ and for any fixed $0\neq s^{\ast}\in S$, the range of the mapping 
\[
 \lambda \longmapsto \eta(\lambda, s^{\ast})
\]
is $]0, +\infty[ \subsetneq [0, +\infty[$, since $0$ is not contained in the range. So the above mapping is not surjective, indicating that it cannot be bijective either. 
Nevertheless, we have 
\[
 \eta(\lambda, s)= H(\lambda \cdot H^{-1}(s))
\]
for all $\lambda, s\in ]0, +\infty[$ with the function 
$
 H(s)= s^{-\delta} 
 \;
 \left(s\in ]0, +\infty[\right). 
$

Finally, assuming $\delta=0$, we have 
\[
 \eta(\lambda, s)= s
\]
for all possible $\lambda$ and $s$. In this case for all $s^{\ast}\in S$, the mapping
\[
 \lambda \longmapsto \eta(\lambda, s^{\ast})
\]
is constant, i.e., cannot be bijective. Further, there is no function $H\colon J\to S$ such that we would have 
\[
 \eta(\lambda, s)= H(\lambda \cdot H^{-1}(s)) 
 \qquad 
 \left(\lambda \in J, s\in S\right). 
\]
Thus, in this case it makes no sense to talk about a representation of the form \eqref{rep1}. Note, however, that the power model used as a starting point in this part reduces in this case to the model 
\[
 \xi_{s}(\lambda x)= \lambda^{\phi(s)}\xi_{s}(x) 
 \qquad 
 \left(x\in I, s\in S\right). 
\]
\end{rem}

{
\begin{cor}
Let $I, J, S\subset \mathbb{R}$ be  intervals with $I\cdot J= \left\{ \lambda x\, \vert \, x\in I, \lambda \in J \right\}\subset I$ and $0, 1\in J$, $0\in S$. Let further $\eta \colon J\times S\to S$ and $g\colon J\to \mathbb{R}$ be functions and assume that the one-parameter family of functions $\xi_{s}\colon I\to \mathbb{R}\, (s\in S)$ fulfills
\begin{equation}\label{cor3_eq}
\xi_{s}(\lambda x)= g(\lambda)\xi_{\eta(\lambda, s)}(x)
\qquad
\left(x\in I, \lambda\in J, s\in S\right).
\end{equation}
Assume that the function $g\colon J\to \mathbb{R}$ is measurable on some subinterval $\tilde{J}$ of $J$ with positive length. Suppose further that $\eta$ is a regular multiplicatively translational mapping, that is 
\[
\eta(\lambda, s)= H(\lambda \cdot H^{-1}(s)) 
\qquad 
\left(\lambda \in J, s\in S\right)
\]
holds with some bijection $H\colon J\to S$ satisfying $H(0)=0$. 
Then there exist a real number $\alpha$, a positive real number $\mu$ and a function $F\colon I\to \mathbb{R}$ such that
\[
g(\lambda)= \alpha \lambda^{\mu}
\qquad
\text{and}
\qquad
\xi_{s}(x)= \alpha x^{\mu}F(x\cdot H^{-1}(s))
\qquad
\left(x\in I, s\in S\right).
\]
Conversely, if we consider the mappings $g\colon J\to \mathbb{R}$, $\eta\colon J\times S\to S$ and $\xi_{s}\colon I\to \mathbb{R}\, (s\in S)$ defined by 
\[
g(\lambda)= \alpha \lambda^{\mu}\;, \qquad 
\eta(\lambda, s)=  H(\lambda \cdot H^{-1}(s)) 
\qquad 
\left(\lambda \in J, s\in S\right)
\]
and 
\[
\xi_{s}(x)= \alpha x^{\mu}F(x\cdot H^{-1}(s))
\qquad
\left(x\in I, s\in S\right), 
\]
where $\alpha \in \mathbb{R}$, $\mu \in ]0, +\infty[$, $H\colon J\to S$ is a bijection with $H(0)=0$ and $F\colon I\to \mathbb{R}$, then identity \eqref{cor3_eq} is fulfilled for all $x\in I, \lambda \in J$ and $s\in S$ only if $\alpha\in \left\{ 0, 1\right\}$. 
\end{cor}
}

\begin{proof}
Due to Proposition \ref{prop2}, we have
\[
\frac{\kappa(h(s)\lambda)}{\kappa(h(s))}= g(\lambda)
\qquad
\left(\lambda\in J, s\in S\right),
\]
that is,
\[
\kappa(h(s)\lambda)= \kappa(h(s))g(\lambda)
\qquad
\left(\lambda\in J, s\in S\right),
\]
or, in other words,
\[
\kappa (\lambda s)= g(\lambda)\kappa(s)
\qquad
\left(\lambda\in J, s\in S\right).
\]
This equation is a multiplicative Pexider equation {on the set $J\times S$}. Further the function $g$ is measurable on some subinterval $\tilde{J}$ of $J$ with positive length. {Thus by \cite{ChuSob17},} the functions $\kappa$ and $g$ are constant multiples of a power function. Therefore, there exist real numbers $\alpha$ and $\mu$ such that
\[
g(\lambda)= \alpha \lambda^{\mu}
\qquad
\left(\lambda\in J\right)
\]
and {by Remark \ref{cor1} we have }
\[
\xi_{s}(x)=
\alpha x^{\mu}F(x\cdot H^{-1}(s))
\qquad
\left(x\in I, s\in S\right).
\]
{
Conversely, let us consider the mappings $g\colon J\to \mathbb{R}$, $\eta\colon J\times S\to S$ and $\xi_{s}\colon I\to \mathbb{R}\, (s\in S)$ defined by 
\[
g(\lambda)= \alpha \lambda^{\mu}\;, \qquad 
\eta(\lambda, s)=  H(\lambda \cdot H^{-1}(s)) 
\qquad 
\left(\lambda \in J, s\in S\right)
\]
and 
\[
\xi_{s}(x)= \alpha x^{\mu}F(x\cdot H^{-1}(s))
\qquad
\left(x\in I, s\in S\right), 
\]
where $\alpha \in \mathbb{R}$, $\mu \in ]0, +\infty[$, $H\colon J\to S$ is a bijection with $H(0)=0$ and $F\colon I\to \mathbb{R}$. Suppose that identity \eqref{cor3_eq} is fulfilled for all $x\in I, \lambda \in J$ and $s\in S$. 
Then 
\[
\xi_{s}(\lambda x)= \alpha (\lambda x)^{\mu}\cdot F(\lambda x \cdot H^{-1}(s)) 
\qquad 
\left(x\in I, \lambda \in J, s\in S\right)
\]
and 
\begin{align*}
g(\lambda)\xi_{\eta(\lambda, s)}(x)&
=
\alpha \lambda^{\mu}\alpha x^{\mu}F(x \cdot H^{-1}(H(\lambda \cdot H^{-1}(s))))\\
&= \alpha^{2} (\lambda x)^{\mu} F(\lambda x \cdot H^{-1}(s))
\qquad (x\in I, \lambda \in J, s\in S). 
\end{align*}
From this however $\alpha= \alpha^{2}$, i.e., $\alpha \in \left\{ 0, 1\right\}$ follows. }
\end{proof}

The cases when $\eta$ is not regular multiplicatively translational, but only multiplicatively translational, can be important from the point of view of applications, 
see \cite{PavelIverson1981} and \cite{HsuIverson2016}.


We would also like to illustrate this case with an example below. Before all this, however, we examine the case when
\[
\eta(\lambda, s)= \lambda^{\theta} s
\qquad
\left(\lambda \in J, s\in S\right).
\]
with a positive $\theta$.

\begin{cor}\label{Cor_shift}
Let $I, J, S\subset \mathbb{R}$ be  intervals with $I\cdot J= \left\{ \lambda x\, \vert \, x\in I, \lambda \in J \right\}\subset I$ and $0, 1\in J$, $0\in S$. Let further $\theta>0$ be arbitrarily fixed, $\gamma \colon J\times S\to \mathbb{R}$ be a function and assume that the one-parameter family of functions $\xi_{s}\colon I\to \mathbb{R}\, (s\in S)$ fulfills the following shift invariance
\[
\xi_{s}(\lambda x)= \gamma(\lambda, s)\xi_{\lambda^{\theta}s}(x)
\qquad
\left(x\in I, \lambda\in J, s\in S\right).
\]
Then there exist functions $f, F\colon I\to \mathbb{R}$ such that
\[
\gamma(\lambda, s)= \frac{f(\lambda s^{\frac{1}{\theta}})}{f(s^{\frac{1}{\theta}})}
\quad
\text{and}
\quad
\xi_{s}(x)= \frac{f(x s^{\frac{1}{\theta}}) }{f(s^{\frac{1}{\theta}})} \cdot F(x^{\theta} s)
\quad
\left(x\in I, \lambda\in J, s\in S\right),
\]
except for those points $s\in S$ for which we have $f(s^{\frac{1}{\theta}})=0$.
\end{cor}

\begin{proof}
Consider the function $H$ defined on $J$ by
\[
H(s)= s^{\theta}
\qquad
\left(s\in S\right).
\]
Then
\[
H(\lambda \cdot H^{-1}(s))= \lambda^{\theta}s
\]
holds for all $\lambda\in J, s\in S$. This means that Proposition \ref{prop2} applies with $h(s)= s^{\frac{1}{\theta}}$ and we have  
\[
\gamma(\lambda, s)= \frac{f(\lambda s^{\frac{1}{\theta}})}{f(s^{\frac{1}{\theta}})}
\]
for all $\lambda\in J, s\in S\setminus \left\{ 0\right\}$. Writing this representation back to the shift-invariance property, we get that
\[
\xi_{s}(\lambda x) = \frac{f(\lambda s^{\frac{1}{\theta}})}{f(s^{\frac{1}{\theta}})} \xi_{\lambda^{\theta}s}(x)
\]
for all $x\in I$ and $\lambda\in J, s\in S\setminus \left\{ 0\right\}$, yielding that
\[
\xi_{s}(x)= \frac{f(x s^{\frac{1}{\theta}}) }{f(s^{\frac{1}{\theta}})} \cdot F(x^{\theta} s)
\qquad
\left(x\in I, s\in S\setminus \left\{ 0\right\}\right)
\]
with an appropriate function $F$.
\end{proof}

In the following statement, we will consider the special shift-invariance corresponding to $\theta=-1$ in the work \cite{IversonPavel1981}. Note that, due to the comments above, Corollary \ref{Cor_shift} cannot be applied in this case, since the mapping
\[
\eta(\lambda, s)= \frac{s}{\lambda}
\]
is not \emph{regular} multiplicatively translational, even though it has the multiplicatively translational property. Accordingly, other ideas and tools are needed to investigate this case.

\begin{prop}
{Let $S, I\subset [0, +\infty[$ and $J\subset ]0, +\infty[$ be intervals such that
$0\in S$, $I\cdot J\subset I$ and $\left\{\frac{s}{\lambda}\, \vert \, s\in S, \lambda\in J \right\}\subset S$.} If the one-parameter family of functions $\xi_{s}\colon I\to \mathbb{R}\, (s\in S)$ fulfills
\[
\xi_{s}(\lambda x)= \lambda \xi_{\frac{s}{\lambda}}(x)
\qquad
\left(x\in I, \lambda\in J, s\in S\right),
\]
then, there exists a real constant $c$ and a function $\varphi\colon \mathbb{R}\to \mathbb{R}$ such that
\[
\xi_{s}(x)=
\begin{cases}
cx, & \text{ if }\, s=0\\
s\cdot \varphi\left(\frac{x}{s}\right), & \text{ if }\, s\neq 0
\end{cases}
\qquad
\left(x\in I, s\in S\right).
\]
\end{prop}

\begin{proof}
With $s=0$ we obtain the multiplicative Pexider equation
\[
\xi_{0}(\lambda x)= \lambda \xi_{0}(x)
\qquad
\left(x\in I, \lambda\in J\right).
\]
Thus
\[
\xi_{0}(x)= cx
\qquad
\left(x\in I\right)
\]
with an appropriate real constant $c$.
Further, with a fixed $s^{\ast}\neq 0$ we deduce that  
\[
\xi_{s^{\ast}}(\lambda x)= \lambda \xi_{\frac{s^{\ast}}{\lambda}}(x)
\qquad
\left(x\in I, \lambda\in J\right).
\]
In other words,
\[
 \xi_{s}(x)= \frac{s}{s^{\ast}}\xi_{s^{\ast}}\left(\frac{s^{\ast}}{s}x\right)
\]
holds for all $x\in I$ and $s\in S\setminus \left\{ 0\right\}$. Thus, if we consider the function $\varphi$ defined on $I$ by 
\[
 \varphi(x)= \frac{1}{s^{\ast}}\xi_{s^{\ast}}(s^{\ast}x) 
 \qquad 
 \left(x\in I\right), 
\]
it can be seen that the statement of the proposition is fulfilled, indeed. 
\end{proof}

In summary, in this subsection we examined the following. Suppose that there is a one-parameter family of functions $\xi$ that satisfies Iverson's law of similarity. In addition, suppose that the function $\eta$ is regular multiplicative translational. The statements above tell us that in this case both the functions $\xi$ and $\gamma$ have some special form. The question of which of these one-parameter families $\xi$ satisfy some psychophysical representations will be answered in one of our future papers.

Finally, we also compare the current results with the results of \cite{DobleHsu2020}. In that paper those sensitivity functions that satisfy Iverson's law of similarity and also admit a subtractive representation were determined.
Note that this automatically entails the fulfillment of certain regularity conditions, which we did not assume in this paper.

Nevertheless, it may be interesting to raise the question of which of the functions $\eta$ in  \cite{DobleHsu2020} have the regular multiplicatively translational property. As we will see below, the boundary conditions appearing in the definition of this notion are quite restrictive.

Indeed, an easy computation shows that all the functions (these appear in Cases II, IV, and V in \cite{DobleHsu2020})
\[
\begin{array}{rcl}
\eta_{1}(\lambda, s)&=& -\dfrac{1}{\kappa}\log \left(\lambda^{-\delta}(e^{-\kappa s}-\beta)+\beta \right) \\
\eta_{2}(\lambda, s)&=& s+\dfrac{\delta}{\kappa}\log(\lambda)\\
\eta_{3}(\lambda, s)&=& \lambda^{-\delta}(s-\varepsilon)+\varepsilon
\end{array}
\]
fulfill the equation
\[
\eta(\lambda \tilde{\lambda}, s)= \eta(\lambda, \eta(\tilde{\lambda}, s))
\qquad
\left(\lambda \in J, s\in J\right).
\]
However, the boundary conditions in the definition are satisfied only if
\begin{itemize}
\item $\beta =1$ or $\delta =0$ in case of the function $\eta_{1}$, i.e.,
\[
\eta_{1}(\lambda, s)= -\dfrac{1}{\kappa} \log \left(\frac{e^{-\kappa s}-1}{\lambda^{\delta}}+1\right)
\;
\text{or}
\;
\eta_{1}(\lambda, s)=s
\quad
\left(\lambda \in J, s\in S\right).
\]
\item $\delta =0$ in case of the function $\eta_{2}$, i.e.,
\[
\eta_{2}(\lambda, s)= s
\qquad
\left(\lambda \in J, s\in S\right).
\]
\item $\delta=0$ or $\varepsilon =0$ in case of the function $\eta_{3}$, i.e.,
\[
\eta_{3}(\lambda, s)= s
\quad
\text{or}
\quad
\eta_{3}(\lambda, s)= s\lambda^{-\delta}
\qquad
\left(\lambda \in J, s\in S\right).
\]
\end{itemize}
Note that these are the multiplicatively translational mappings, of which the function
\[
\eta(\lambda, s)=s
\]
is \emph{not} regular. In this case, however, Iverson's law of similarity is a multiplicative type equation. Thus if we assume e.g.~that for all fixed $s\in S$, $\eta$ is monotonic on $I$, then we obtain that
\[
\xi_{s}(x)= c(s)x^{\varphi(s)}
\qquad
\left(x\in I, s\in S\right)
\]
with some appropriate functions $c, \varphi \colon S\to \mathbb{R}$.

Further, the mapping
\[
\eta(\lambda, s)= s \cdot \lambda^{-\delta}
\]
is regular multiplicatively translational, and this case corresponds to the shift-invariance dealt with in Corollary \ref{Cor_shift}.

Finally, if  we consider only regular multiplicatively translational mappings, then in the first case of $\eta_{1}$ we have
\[
H(s)= -\dfrac{1}{\kappa} \log \left(\frac{\alpha}{s^{\delta}}+1\right)
\qquad
\left(s\in S\right).
\]
Since now we have $[0, 1]\subset J$, the parameter $\delta$ must be negative. So we may write $-\delta$ in place of $\delta$ (with some positive $\delta$). Further, we set $\kappa= -\log(\alpha+1)$, which guarantees $H(1)=1$. The latter is only a technical condition that makes calculations easier. This brings us to the last case, i.e., when
\[
H(s)= \dfrac{\log(\alpha s^{\delta}+1)}{\log(\alpha+1)}
\qquad
\left(s\in S\right),
\]
in other words,
\[
\eta(\lambda, s)= \dfrac{\log\left(\lambda^{\delta}\left(e^{\log(\alpha+1)s}-1\right)+1\right)}{\log(\alpha+1)}
\qquad
\left(\lambda \in J, s\in S\right).
\]
Note that in this case, we have
\[
H^{-1}(s)= \left(\frac{\exp(\log(\alpha+1)s)-1}{\alpha}\right)^{\frac{1}{\rho}}
\qquad
\left(s\in S\right).
\]
Given the argument presented at the beginning of this section,
\[
\xi_{s^{\ast}}\left(\frac{H^{-1}(s)}{H^{-1}(s^{\ast})} x\right)= \gamma\left( \frac{H^{-1}(s)}{H^{-1}(s^{\ast})}, s^{\ast}\right) \xi_{s}(x)
\qquad
\left(x\in I, s \in S\right)
\]
with some fixed $s^{\ast}\in S$. Assume now that there exists $\bar{s}\in S$ and a subinterval $\bar{S}\subset S$ such that $\gamma(\lambda, \bar{s})\neq 0$ for all
$\lambda\in J$. Then
\[
\xi_{s}(x)= \dfrac{\Phi(H^{-1}(s)x)}{g(H^{-1}(s))}
\qquad
\left(x\in I, s \in S\right)
\]
holds with some appropriate functions $\Phi$ and $g$.
The question of which one-parameter functions of this form satisfy a psychometric representation may also be answered in a future paper.

\section*{Declarations}

\begin{description}
\item[Funding] The research of E. Gselmann has been supported by project no. K134191
that has been implemented with the support provided by the National Research,
Development and Innovation Fund of Hungary, financed under the K\_20 funding
scheme. Y.-F. Hsu's work was supported in part by the Ministry of Science and Technology of Taiwan (R.O.C.) Grant MOST-111-2410-H-002-166.
\item[Conflict of interest/Competing interests] The authors have no conflicts of interest
to declare that are relevant to the content of this article.
\item[Ethics approval and consent to participate] Not applicable.
\item[Consent for publication] We hereby provide consent for the publication of the
manuscript detailed above, including any accompanying images or data contained
within the manuscript that may directly or indirectly disclose our identity.
\item[Declaration of generative AI] We declare that have not used any AI tools or technologies to prepare this manuscript. 
\item[Data availability] Not applicable.
\item[Materials availability] Not applicable.
\item[Code availability] Not applicable.
\item[Author contribution] The authors contributed equally.
\end{description}

%

\end{document}